\documentclass[11pt,reqno]{amsart}
\usepackage{amssymb}
\usepackage{amsmath, amssymb, amsthm}
\usepackage{mathrsfs}

\newtheorem{theorem}{Theorem}[section]
\newtheorem{definition}{Definition}[section]
\newtheorem{lemma}{Lemma}[section]
\newtheorem{remark}{Remark}[section]

\newtheorem{corollary}{Corollary}[section]
\numberwithin{equation}{section}

\begin{document}
\title[Navier-Stokes equations]{   On classical solutions for viscous polytropic fluids with degenerate viscosities  and vacuum}

\author{yachun li}
\address[Y. C. Li]{Department of Mathematics, MOE-LSC, and SHL-MAC, Shanghai Jiao Tong University,
Shanghai 200240, P.R.China} \email{\tt ycli@sjtu.edu.cn}

\author{ronghua pan  }
\address[R. H. Pan]{School of Mathematics, Georgia Tech
Atlanta 30332, U.S.A.} \email{\tt panrh@math.gatech.edu }

\author{Shengguo Zhu}
\address[S. G.  Zhu]{Department of Mathematics, Shanghai Jiao Tong University,
Shanghai 200240,  P.R.China;
School of Mathematics, Georgia Institute of Technology, Atlanta 30332, U.S.A.}
\email{\tt zhushengguo@sjtu.edu.cn}

\begin{abstract}
In this paper, we consider the three-dimensional  isentropic Navier-Stokes equations for compressible fluids allowing initial vacuum when viscosities depend on density in a superlinear power law. We introduce the notion of regular solutions and prove the local-in-time well-posedness of solutions with arbitrarily large initial data  and vacuum in this class, which is a long-standing open problem due to the very high degeneracy caused by vacuum.  Moreover,  for certain classes of initial data with local vacuum, we show that the regular solution that we obtained will break down in finite time,  no matter how small and smooth the initial data are.
\end{abstract}

\date{April 3, 2015}
\subjclass[2010]{Primary: 35B40, 35A05, 76Y05; Secondary: 35B35, 35L65} \keywords{Compressible Navier-Stokes,  vacuum, degenerate viscosity,  classical solutions,    blowup}.\\

\maketitle

\section{Introduction}\ \\

In this paper, we investigate the local-in-time well-posedness and formation of singularities of classical solutions to the compressible Navier-Stokes equations  for isentropic flows when viscosity coefficients, shear and bulk, are both degenerate and the initial data are arbitrarily large with possible vacuum states.  The system of compressible isentropic Navier-Stokes equations in $\mathbb{R}^3$ reads as
\begin{equation}
\label{eq:1.1}
\begin{cases}
\rho_t+\text{div}(\rho u)=0,\\[4pt]
(\rho u)_t+\text{div}(\rho u\otimes u)
  +\nabla
   P =\text{div} \mathbb{T}.
\end{cases}
\end{equation}
We look for local classical solution with initial data
\begin{equation} \label{eq:2.211m}
(\rho,u)|_{t=0}=(\rho_0(x),u_0(x)),\quad x\in \mathbb{R}^3,
\end{equation}
and far field behavior
\begin{equation} \label{eq:2.211}
(\rho,u)\rightarrow (\overline{\rho},0) \quad \text{as } \quad |x|\rightarrow +\infty,\quad t> 0,
\end{equation}
where $\overline{\rho} \geq 0$ is a  constant.

In system (\ref{eq:1.1}), $x=(x_1,x_2,x_3)\in \mathbb{R}^3$, $t\geq 0$ are space and time variables, respectively. $\rho$ is the density, and $u=\left(u^{(1)},u^{(2)},u^{(3)}\right)^\top\in \mathbb{R}^3$ is the velocity of the fluid. We assume that the pressure  $P$ satisfies
\begin{equation}
\label{eq:1.2}
P=A\rho^{\gamma}, \quad \gamma> 1,
\end{equation}
where $A$ is a positive constant, $\gamma$ is the adiabatic exponent. $\mathbb{T}$ denotes the viscosity stress tensor with the  form
\begin{equation}
\label{eq:1.3}
\begin{split}
&\mathbb{T}=\mu(\rho)(\nabla u+(\nabla u)^\top)+\lambda(\rho)\, \text{div}u\, \mathbb{I}_3,\\
\end{split}
\end{equation}
 where $\mathbb{I}_3$ is the $3\times 3$ identity matrix. We assume in this paper that
\begin{equation}
\label{fandan}
\mu(\rho)=\alpha\rho^\delta,\quad \lambda(\rho)=\beta\rho^\delta,
\end{equation}
 $\mu(\rho)$ is the shear viscosity coefficient, $\lambda(\rho)+\frac{2}{3}\mu(\rho)$ is the bulk viscosity coefficient, $\alpha$ and $\beta$ are both constants satisfying the following physical constraints
 \begin{equation}\label{10000}\alpha>0,\quad  2\alpha+3\beta\geq 0.
 \end{equation}
Furthermore, in this paper, we require that the constant $\delta$ satisfies
 \begin{equation}\label{100001} 1<\delta \leq \min \Big\{ \frac{\gamma+1}{2},3\Big\}.
 \end{equation}

In the theory of gas dynamics, when we derive the compressible Navier-Stokes equations from the Boltzmann equations through the Chapman-Enskog expansion to the second order, cf. Chapman-Cowling \cite{chap} and Li-Qin \cite{tlt}, we find that the viscosity coefficients are not constants but functions of absolute temperature in a power law. For isentropic flow,  this dependence on the absolute temperature, from the laws of Boyle and Gay-Lussac, can be reduced to the dependence on the density, that is to say, the viscosity coefficients are proportional to powers of density, see \cite{tlt}.

When $\inf_x {\rho_0(x)}>0$, it is well-known that the local existence of classical solutions for \eqref{eq:1.1}--\eqref{eq:2.211} can be obtained by a standard Banach fixed point argument due to the contraction property of the solution operators of the linearized problem, c.f. Nash \cite{nash}. However, this approach does not work when $\inf_x {\rho_0(x)}=0$, which occurs when some physical requirements are imposed, such as finite total initial mass,  finite total initial energy, or vacuum appearing locally in some open sets.

When viscosity coefficients $\mu$ and $\lambda$ are both constants, for the existence of  solutions of the three-dimensional isentropic flows with generic data and vacuum, the main breakthrough is due to Lions \cite{lions}, where he established the global existence of weak solutions provided that $\gamma >\frac95$. This result is improved to the case $\gamma>\frac32$ by Feireisl-Novotn\'{y}-Petzeltov\'{a}\cite{fu1}, and even to
 non-isentropic flows in \cite{fu2, fu3}. However, the regularities of these weak solutions are fairly low, and the
uniqueness problem is widely open. On the other hand, the local well-posedness problem in higher
regularity class with possible vacuum initial data requires comprehensive treatment on the degeneracy in time evolution in momentum equations $(\ref{eq:1.1})_2$. Since the leading coefficient of  $u_{t}$ is $\rho$ which vanishes at vacuum, there are infinitely many ways to define velocity (if it exists) when
vacuum appears.  Mathematically, this degeneracy brings forth an essential difficulty in determining the velocity when
vacuum occurs, since it is difficult to find a reasonable way to extend the definition of velocity into vacuum region.  Physically, it is not clear how to define the fluid velocity when there is no fluid at vacuum. A reasonable framework was proposed by  Cho-Choe-Kim  \cite{CK3, CK, guahu} through a  {\it compatibility condition}
\begin{equation}\label{constant}
-\text{div} \mathbb{T}_{0}+\nabla P(\rho_0)=\sqrt{\rho_{0}} g
\end{equation}
for some $g\in L^2$. In turn, a local theory of strong or classical solutions with initial vacuum  was established successfully for three-dimensional case; see also Duan-Zheng-Luo \cite{yuxi} and  Luo \cite{luoluo} for two-dimensional case. More recently, Huang-Li-Xin \cite{HX1} extended this local classical solution (\cite{guahu}) to  a global one with small energy and vacuum for isentropic flow in $\mathbb{R}^3$.  In these results, the uniform ellipticity of viscosity  plays a key role in the a prior estimates of  higher order terms of velocity $u$.

When viscosity coefficients $\mu$ and $\lambda$ are both density-dependent,  system \eqref{eq:1.1} has received extensive attentions in recent years. Instead of the uniform elliptic structure, the viscosity degenerates at vacuum, which raises the difficulty of the problem to another level.  A remarkable
discovery of a new mathematical entropy function  was made by Bresch-Desjardins \cite{bd3} for
 $\lambda(\rho)$ and $\mu(\rho)$ satisfying the relation
\begin{equation}\label{bds}\lambda(\rho)=2(\mu'(\rho)\rho-\mu(\rho)).
\end{equation}
For example, in our problem, such an entropy structure (\ref{bds}) exists  when $\beta=2\alpha(\delta-1)>0$.
This new entropy offers a nice estimate
$$
\mu'(\rho)\frac{\nabla \rho}{\sqrt{\rho}}\in L^\infty ([0,T];L^2(\mathbb{R}^d))
$$
provided that $ \mu'(\rho_0)\frac{\nabla \rho_0}{\sqrt{\rho_0}}\in L^2(\mathbb{R}^d)$ for any $d\geq 1$.
This observation plays an important role in the development of the global existence of weak solutions with vacuum for system (\ref{eq:1.1}) and related models, see Bresh-Dejardins \cite{bd6, bd7}, Mellet-Vassuer \cite{vassu}, and Vassuer-Yu \cite{vy}, and some other interesting results, c.f. \cite{ bd2, bd,   hailiang, taiping, sz2, zyj, tyc2}.  However, we remark that, in spite of these significant achievements mentioned above,  the local well-posedness of classical solutions in multi-dimensions with vacuum  is still open due to the high degeneracy of this system near vacuum region.  Our existence result in this paper is a solid step toward this direction.

From (\ref{eq:1.2})-(\ref{eq:1.3}), we notice that
$$
\text{div}\mathbb{T}=-\rho^\delta Lu+\nabla \rho^\delta \cdot \mathbb{S}(u),
$$
where the so-called Lam\'e operator $ L$ and  the operator $\mathbb{S} $ are given by
\begin{equation}\label{str2}
Lu=-\alpha\triangle u-(\alpha+\beta)\nabla \mathtt{div}u,\quad  \mathbb{S}(u)=\alpha(\nabla u+(\nabla u)^\top)+\beta\, \mathtt{div}u\, \mathbb{I}_3.
\end{equation}
When constructing the classical solutions  with arbitrarily large data and vacuum,   in addition to the issue shown above for the constant viscosity case on the degeneracy in time evolution in momentum equations,  there are still   two new significant  difficulties due to  the  appearance of vacuum.
The first one lies in the strong degeneracy of the elliptic   operator $\text{div}\mathbb{T}$ caused by vacuum. We note that under assumptions \eqref{fandan}-\eqref{100001},  viscosity coefficients vanish as density function connects to vacuum continuously in some open sets.  This degeneracy gives rise to some difficulties in our analysis because of the less regularizing effect of the viscosity on the solutions, and is one of the major obstacles preventing us from utilizing a similar remedy proposed by Cho et. al. \cite{CK3, CK, guahu} for the case of constant viscosity coefficients.
The second one lies in   the extra nonlinearity for the variable coefficients in $\text{div}\mathbb{T}$ due to (\ref{fandan}). We emphasize here that, unlike the constant viscosity  case, the elliptic part  $\text{div} \mathbb{T}$ is not always a  good term in the  regularity analysis for the higher order terms of the velocity. For example, if we want to get the estimate on $\|\nabla^3 u\|_{L^\infty([0,T];L^2(\mathbb{R}^3))}$ independent of the lower bound of the initial density, we need to deal with  an extra nonlinear term
$$\text{div}\big( \partial^\zeta_x\rho^\delta \big(\alpha(\nabla u+(\nabla u)^\top)+\beta\mathtt{div}u\mathbb{I}_3\big) \big),$$
where $\zeta=(\zeta_1,\zeta_2,\zeta_3)$ is a multi-index with $|\zeta|=\zeta_1+\zeta_2+\zeta_3=3$.
Therefore,  much attention will be  paid to control this strong nonlinearity.

In order to overcome these difficulties,  some  new ideas have to be introduced.
In \cite{sz3}, we obtained   the existence of the unique local classical solutions for system (\ref{eq:1.1})  with  vacuum far field under the assumption
\begin{equation}\label{case1}\delta=1, \quad \alpha>0,\quad \alpha+\beta\geq 0
\end{equation}
for two-dimensional space,  aiming at the application to shallow water models.
In \cite{sz3}
we observed that, assuming $\rho>0$, the momentum equations can be rewritten as
\begin{equation}\label{str1}
u_t+u\cdot\nabla u +\frac{2A\gamma}{\gamma-1}\rho^{\frac{\gamma-1}{2}}\nabla\rho^{\frac{\gamma-1}{2}}+Lu=\Big(\frac{\nabla \rho}{\rho}\Big) \cdot \mathbb{S}(u).
\end{equation}
Then, (\ref{str1}) implies that if we can control the special  source term  $\Big(\frac{\nabla \rho}{\rho}\Big) \cdot  \mathbb{S}(u)$ when vacuum appears,  the velocity $u$ of  the fluid  can be governed by a strong parabolic system. However, this result only allows vacuum at the far field. The corresponding problem with vacuum appearing in some open sets, or even at a single point is still unsolved.

Motivated by this observation in \cite{sz3}, when $\rho> 0$,  for our case under the assumption (\ref{100001}), instead of (\ref{str1}), $(\ref{eq:1.1})_2$ can be rewritten as
\begin{equation}\label{jgh}
\begin{split}
&u_t+u\cdot\nabla u +\rho^{\delta-1}Lu\\
=&-\frac{2A\gamma}{\delta-1}(\rho^{\frac{\delta-1}{2}})^{\frac{2\gamma-\delta-1}{\delta-1}}\nabla \rho^{\frac{\delta-1}{2}}+\frac{2\delta}{\delta-1}\rho^{\frac{\delta-1}{2}}\nabla \rho^{\frac{\delta-1}{2}} \cdot \mathbb{S}(u).
\end{split}
\end{equation}
Then if we pass to the limit as   $\rho\rightarrow 0$ on both sides of (\ref{jgh}), we formally have
\begin{equation}\label{zhenkong}
 u_t+ u\cdot \nabla u=0 \quad \text{when}\ \rho=0.
\end{equation}
So (\ref{jgh})-(\ref{zhenkong}) imply that actually the velocity $u$ can be governed by a nonlinear degenerate  parabolic system  when vacuum appears in some open sets or at the far field, which is essentially different from the parabolic system in (\ref{str1}) in the sense of mathematical structure.
Based on this observation, we introduce a proper class of solutions and  prove the local-in-time well-posedness of solutions with arbitrarily large data and vacuum in this class for  system (\ref{eq:1.1}) using
a new approach which bridges the parabolic system (\ref{jgh}) when $\rho>0$, and the hyperbolic system (\ref{zhenkong}) when $\rho=0$.

In order to present our results clearly,  we first introduce the following definition of regular solutions to (\ref{eq:1.1})-(\ref{eq:2.211}).
\begin{definition}[\text{\textbf{Regular solutions}}]\label{d1}\ Let $T> 0$ be a finite constant.  $(\rho,u)(x,t)$ is called a regular solution  in $ \mathbb{R}^3\times[0, T]$ to the  Cauchy problem (\ref{eq:1.1})-(\ref{eq:2.211}) if
\begin{equation*}\begin{split}
&(\textrm{A})\quad  (\rho,u) \ \text{satisfies the Cauchy problem (\ref{eq:1.1})-(\ref{eq:2.211}) in the sense of distribution};\\
&(\textrm{B})\quad  \rho\geq 0, \   \rho^{\frac{\delta-1}{2}}-\overline{\rho}^{\frac{\delta-1}{2}} \in C([0,T]; H^3), \ (\rho^{\frac{\delta-1}{2}})_t \in C([0,T]; H^2); \\
&(\textrm{C})\quad u\in C([0,T]; H^{s'}) \cap L^\infty([0,T]; H^3),\  \rho^{\frac{\delta-1}{2}} \nabla^4 u \in L^2([0,T] ; L^2), \\
&\qquad \  u_t \in C([0,T]; H^1)\cap L^2([0,T] ; D^2);\\
&(\textrm{D})\quad u_t+u\cdot\nabla u =0\quad  \text{holds when }\ \rho(t,x)=0,
\end{split}
\end{equation*}
for any constant  $s' \in[2,3)$.
\end{definition}
Here and throughout this paper, we adopt the following simplified notations, most of them are for the standard homogeneous and inhomogeneous Sobolev spaces:
\begin{equation*}\begin{split}
& |f|_p=\|f\|_{L^p(\mathbb{R}^3)},\quad \|f\|_s=\|f\|_{H^s(\mathbb{R}^3)},\quad  |f|_2=\|f\|_0=\|f\|_{L^2(\mathbb{R}^3)}, \\[8pt]
&D^{k,r}=\{f\in L^1_{loc}(\mathbb{R}^3):  |\nabla^kf|_{r}<+\infty\},\quad D^k=D^{k,2},\quad |f|_{D^{k,r}}=\|f\|_{D^{k,r}(\mathbb{R}^3)}\ (k\geq 2),\\[8pt]
&  D^{1}=\{f\in L^6(\mathbb{R}^3):  |\nabla f|_{2}<\infty\},\quad|f|_{D^{1}}=\|f\|_{D^{1}(\mathbb{R}^3)}.
\end{split}
\end{equation*}
 A detailed study of homogeneous Sobolev spaces  can be found in \cite{gandi}.
\begin{remark}\label{mark2}
This notion of regular solution for compressible Navier-Stokes equations  was first introduced by  Yang-Zhu \cite{tyc2} for one-dimensional case. We defined the regular solutions in \cite{sz3} for the case $\delta=1$.  Compared with \cite{sz3}, a significant difference is that the admissible initial data in this paper are much broader.  Actually in \cite{sz3}, in order to make sure that  the source term $ \Big(\frac{\nabla \rho}{\rho}\Big) \cdot  \mathbb{S}(u) $ appearing in the equations (\ref{str1}) is well defined in $H^2$ space, we need 
$$
\rho^{\frac{\gamma-1}{2}} \in H^3(\mathbb{R}^2),\quad \frac{\nabla \rho}{\rho} \in L^6 \cap D^1 \cap D^2(\mathbb{R}^2),
$$
which means that the vacuum must and only appear in the far field.
In this paper, we only need
$$
\rho^{\frac{\delta-1}{2}}-\overline{\rho}^{\frac{\delta-1}{2}} \in H^3(\mathbb{R}^3).
$$
So the vacuum can appear in any open set or in the far field.\end{remark}

Now we give the main existence result of this paper.
\begin{theorem}\label{th2}
If the initial data $( \rho_0, u_0)$ satisfy the following regularity conditions:
\begin{equation}\label{th78}
\begin{split}
\rho_0\geq 0,\quad  \left( \rho^{\frac{\delta-1}{2}}_0-\overline{\rho}^{\frac{\delta-1}{2}}, u_0\right)\in H^3(\mathbb{R}^3),
\end{split}
\end{equation}
then there exists a positive  time $T_*$ and a unique regular solution $(\rho, u)(x,t)$ in $\mathbb{R}^3\times[0, T_*]$ to  Cauchy problem (\ref{eq:1.1})-(\ref{eq:2.211}).
\end{theorem}

\begin{remark}\label{re3}
The weak smoothing effect of the velocity $u$ in positive time $t\in [\tau,T_*]$, $\forall \tau \in (0,T_*)$, tells us that the regular  solution obtained in Theorem \ref{th2} is indeed a classical one in $(0,T_*]\times \mathbb{R}^3$. The details can be found in the Appendix.
\end{remark}

\begin{remark}\label{hg} We can also consider the flow with external force  in the momentum equations $(\ref{eq:1.1})_2$ such as
\begin{equation}\label{jia}
(\rho u)_t+\text{div}(\rho u\otimes u)
  +\nabla
   P =\text{div} \mathbb{T}+\rho h.
\end{equation}
 Assume that
\begin{equation}\label{jia1}
h\in C([0,T];H^1(\mathbb{R}^3))\cap L^2([0,T];H^3(\mathbb{R}^3))
\end{equation}
for some $T>0$, then  the same conclusion as in Theorem \ref{th2} still holds if we  replace the condition $(C)$ in Definition \ref{d1} with
$$
 u_t+ u\cdot \nabla u=h \quad \text{when}\ \rho=0.
$$
To achieve this result, one only needs to make minor modifications on our proof for Theorem 1.1. We will
not go into details on this matter.
\end{remark}

As a direct consequence of Theorem \ref{th2} and by the standard theory of quasi-linear hyperbolic equations, we have the following additional regularities for some specially choices of power $\delta$.
\begin{corollary}\label{co2} Let
$1< \delta \leq \min \Big\{\frac{5}{3}, \frac{\gamma+1}{2}\Big\},\ \text{or} \ \delta=2 \ (\gamma\geq 3).$ If  the initial data $( \rho_0, u_0)$ satisfy (\ref{th78}),
then the regular solution $( \rho, u)$ obtained in Theorem \ref{th2} also satisfies
\begin{equation}\label{regco}\begin{split}
& \rho-\overline{\rho}\in C([0,T_*);H^3),\quad  \rho_t \in C([0,T_*);H^2).
\end{split}
\end{equation}
\end{corollary}

Compared with previous results for $\delta=0$ \cite{guahu} and $\delta=1$ \cite{sz3}, our analysis in the proof of Theorem  \ref{th2}  is quite different.  From the observations shown in  (\ref{jgh})-(\ref{zhenkong}), the behavior of the velocity $u$ near vacuum under the assumption (\ref{100001}) is more hyperbolic than the cases in  \cite{guahu} and  \cite{sz3}.
This phenomenon leads us to developing a different approach in the analysis on the regularity  of $u$. As a matter of fact, in  \cite{CK3, CK, guahu, sz3}, the uniform ellipticity of the Lam\'e operator $ L$ defined in (\ref{str2}) plays an essential role. Therefore, one can use standard elliptic theory to estimate $|u|_{D^{k+2}}$ ($0\leq k\leq 2$) by the $D^k$-norm of all other terms in momentum equations $(\ref{eq:1.1})_2$ for the case $\delta=0$ in  \cite{CK3, CK, guahu}, or in (\ref{str1}) for the case $\delta=1$ in \cite{sz3, sz33}. However,  for our case, it is clear from (\ref{jgh}) that the standard parabolic theory only offers the
  following weighted estimate:
$$
\rho^{\frac{\delta-1}{2}} \partial^\zeta_x u\in L^2([0,T];L^2),\quad \text{for} \quad |\zeta|=1,2,3,4.
$$
This estimate alone is not enough for the  regularity analysis on  $u$ and $\rho$.
In order to obtain the desired estimate
$$
u \in L^\infty([0,T],H^3)\cap C([0,T]; H^{s'})\quad \text{for\  any}\  s'\in [2,3)
$$
shown in Definition \ref{d1}, we turn to the help of symmetric hyperbolic structure
$$
u_t+u\cdot \nabla u.
$$
These estimates are accomplished along with a vanishing viscosity limit argument.\\

Another challenging question in the theory of fluid dynamics is the problem of global regularity. This problem is extremely difficult for three-dimensional compressible Navier-Stokes equations, especially when the initial density contains vacuum. When viscosity coefficients are constants, global regularity was achieved when initial energy is small, see \cite{Hoff} for the case away from vacuum, and \cite{HX1} for the case allowing possible initial vacuum for the solutions in the class of
\cite{CK3}. In particular, the results of \cite{HX1} showed that no singularity will form in finite time for smooth solutions in the class of \cite{CK3} if the initial energy is small. On the other hand, there are a lot of finite time blowup results obtained, say \cite{y1, luoluo, olg1, zx, zwy, tyc2}, for various solution classes or conditions. However, it is not clear yet if such solutions exist locally in time.  For our case, when the viscosity coefficients satisfy the conditions (\ref{fandan})-(\ref{100001}), we will show that  the  solution  we obtained will break down in finite time for certain classes of initial data with
local vacuum, no matter how small the initial energy is. This is in sharp contrast to the case of constant viscosity as shown in  \cite{CK3, guahu, HX1}.  This is because our problem behaves more closely to compressible Euler equations in vacuum region, due to   strong degeneracy of the viscosity.

In fact, we will present two scenarios for singularity development in finite time from local vacuum initial data.  The first one is motivated by Xin-Yan \cite{zwy}, called the initial data with isolated mass group.

\begin{definition}[\textbf{Isolated mass group}]\label{local}  $(\rho_0(x),u_0(x))$ is said to have an isolated mass group $(A_0,B_0)$, if there are two smooth,  bounded and connected open sets  $A_0 \subset \mathbb{R}^3$ and $B_0\subset \mathbb{R}^3$ satisfying
\begin{equation} \label{eq:12131}
\begin{cases}
\displaystyle
\overline{A}_0 \subset B_0 \subseteq  B_{R_0} \subset \mathbb{R}^3;         \\[10pt]
\displaystyle
\rho_0(x)=0,\ \forall \  x\in B_0 \setminus A_0, \quad \int_{A_0} \rho_0(x) \text{d}x>0; \\[10pt]
\displaystyle
 u_0(x)|_{\partial A_0}=\overline{u}_0
\end{cases}
\end{equation}
for some positive constant $R_0$ and  constant vector $\overline{u}_0 \in \mathbb{R}^3$, where $B_{R_0}$ is  the ball centered at the origin with radius $R_0$.
\end{definition}

Our first blowup result shows that an isolated mass group in initial data guarantees the finite time singularity formation of regular solutions.
\begin{theorem}[\textbf{Blow-up by isolated mass group}]\label{coo2}\
 If the initial data $(\rho_0,u_0)(x)$ have an isolated mass group $(A_0,B_0)$,
then the  regular solution $(\rho, u)(x,t)$ in $\mathbb{R}^3\times[0, T_m]$ obtained in Theorem \ref{th2} with maximal existence time $T_m$ blows up in finite time, i.e.,
$\
T_m<+\infty.
$
\end{theorem}

For the second scenario, we explore the hyperbolic structure for the system in vacuum region. For this purpose, we  introduce the following concept.
\begin{definition}[\textbf{Hyperbolic singularity set}]\label{bugers}
The non-empty open set $V \subset \mathbb{R}^3$ is  called a hyperbolic singularity set of $(\rho_0,u_0)(x)$, if $V$ satisfies
\begin{equation} \label{eq:12131ss}
\begin{cases}
\displaystyle
\rho_0(x)=0, \ \forall \ x\in V;\\[10pt]
\displaystyle
 Sp(\nabla u_0) \cap \mathbb{R}^-\neq\  \emptyset,\quad  \forall  \ x \in V,
\end{cases}
\end{equation}
where $Sp(\nabla u_0(x))$ denotes the spectrum of the matrix $\nabla u_0(x)$.
\end{definition}
This definition is inspired by  the global existence theory of classical solutions to the compressible Euler equations in \cite{GS} and \cite{danni}, for initial data satisfying the following conditions
 \begin{itemize}
\item $\rho_0$ is small and  compactly supported,
\item$
Sp(\nabla u_0) \cap \mathbb{R}^-=\  \emptyset,\quad  \forall  \ x\in \mathbb{R}^3.
$
\end{itemize}

Our second blowup result confirms that hyperbolic singularity set does generate singularity from local regular solutions in finite time.
\begin{theorem}[\textbf{Blow-up by hyperbolic singularity set}]\label{coo22}\
 If the initial data $(\rho_0,u_0)(x)$ have a non-empty hyperbolic singularity set $V$, then
the  regular solution $(\rho, u)(x,t)$ on $\mathbb{R}^3\times[0, T_m]$ obtained in Theorem \ref{th2} with maximal existence time $T_m$ blows up in finite time, i.e.,
$
T_m<+\infty.
$
\end{theorem}

We now outline the organization of the rest of this paper. In Section $2$, we list some important lemmas that will be used frequently in our proof. In Section $3$,  we first reformulate our problem  into a simpler form in terms of some new variables, and then  we give the proof of the local existence of strong solutions to this reformulated problem, which is achieved in four steps: (1) we construct global approximate solutions for a
specially designed linearized problem with an artificial viscosity $\eta^2 Lu$ in  momentum equations; (2) we establish the a priori estimates independent of artificial viscosity coefficient $\eta$; (3) we then pass to the limit $\eta\rightarrow 0$ to recover the solution of this linearized problem allowing degeneracy in the elliptic operator appearing in momentum equations; (4) we prove the unique solvability of the reformulated problem through a standard iteration process. Section $4$ is devoted to proving the finite time blowup shown in Theorems \ref{coo2}-\ref{coo22}. Section 5 is the appendix where we prove the fact that the regular solution we obtained in Theorem \ref{th2} is indeed a classical one in $(0,T_{*}]$.

Finally, we remark that our framework  in this paper is applicable to other physical dimensions, say 1 and 2, with some minor modifications. This is clear from the analysis carried out in the following sections.

\section{Preliminaries}

In this section, we list some important lemmas  that will be frequently used in our proofs.
The first one is the  well-known Gagliardo-Nirenberg inequality.
\begin{lemma}\cite{oar}\label{lem2as}\
For $p\in [2,6]$, $q\in (1,+\infty)$, and $r\in (3,+\infty)$, there exists some generic constant $C> 0$ that may depend on $q$ and $r$ such that for
$$f\in H^1(\mathbb{R}^3)\quad \text{and} \quad  g\in L^q(\mathbb{R}^3)\cap D^{1,r}(\mathbb{R}^3),$$
 we have
\begin{equation}\label{33}
\begin{split}
&|f|^p_p \leq C |f|^{(6-p)/2}_2 |\nabla f|^{(3p-6)/2}_2,\\[8pt]
&|g|_\infty\leq C |g|^{q(r-3)/(3r+q(r-3))}_q |\nabla g|^{3r/(3r+q(r-3))}_r.
\end{split}
\end{equation}
\end{lemma}
Some commonly used versions of this inequality are
\begin{equation}\label{ine}\begin{split}
|u|_6\leq C|u|_{D^1},\quad |u|_{\infty}\leq C| u|^{\frac{1}{2}}_{6}| \nabla u|^{\frac{1}{2}}_{6}, \quad |u|_{\infty}\leq C\|u\|_{W^{1,r}}.
\end{split}
\end{equation}

The second lemma is some compactness results obtained via the Aubin-Lions Lemma.
\begin{lemma}\cite{jm}\label{aubin} Let $X_0\subset X\subset X_1$ be three Banach spaces.  Suppose that $X_0$ is compactly embedded in $X$ and that $X$ is continuously embedded in $X_1$. Then the following statements hold:

\begin{enumerate}
\item If $G$ is bounded in $L^p(0,T;X_0)$ for $1\leq p < +\infty$, and $\frac{\partial G}{\partial t}$ is bounded in $L^1(0,T;X_1)$, then $G$ is relatively compact in $L^p(0,T;X)$.

\item If $G$ is bounded in $L^\infty(0,T;X_0)$  and $\frac{\partial G}{\partial t}$ is bounded in $L^p(0,T;X_1)$ for $p>1$, then $G$ is relatively compact in $C(0,T;X)$.
\end{enumerate}
\end{lemma}

The following four lemmas contain some Sobolev inequalities on the product estimates, the interpolation estimates,  the composite function estimates, etc.,  which can be found in many literatures, say Majda \cite{amj}.  We omit the proofs of them.

\begin{lemma}\cite{amj}\label{zhen1}
Let constants $r$, $a$ and $b$ satisfy the relation
$$\frac{1}{r}=\frac{1}{a}+\frac{1}{b},\quad \text{for} \quad 1\leq a,\ b, \ r\leq +\infty.$$  $ \forall s\geq 1$, if $f, g \in W^{s,a} \cap  W^{s,b}(\mathbb{R}^3)$, then
\begin{equation}\begin{split}\label{ku11}
&|\nabla^s(fg)-f \nabla^s g|_r\leq C_s\big(|\nabla f|_a |\nabla^{s-1}g|_b+|\nabla^s f|_b|g|_a\big),
\end{split}
\end{equation}
\begin{equation}\begin{split}\label{ku22}
&|\nabla^s(fg)-f \nabla^s g|_r\leq C_s\big(|\nabla f|_a |\nabla^{s-1}g|_b+|\nabla^s f|_a|g|_b\big),
\end{split}
\end{equation}
where $C_s> 0$ is a constant depending only on $s$, and $\nabla^s f$ ($s>1$) stands for the set of  all partial derivatives $\partial^\zeta_x f$  of order $|\zeta|=s$.
\end{lemma}

\begin{lemma}\cite{amj}\label{gag113}
If  functions $f,\ g \in H^s$ and $s>\frac{3}{2}$, then  $f g \in H^s$,  and  there exists  a constant $C_s$
 depending only on $s$ such that
$$
\|fg\|_{s} \leq C_s \|f\|_s \|g\|_s.
$$
\end{lemma}

\begin{lemma}\cite{amj}\label{gag111}
If  $u\in H^s$, then for any $s'\in[0,s]$,  there exists  a constant $C_s$ depending only on $s$ such that
$$
\|u\|_{s'} \leq C_s \|u\|^{1-\frac{s'}{s}}_0 \|u\|^{\frac{s'}{s}}_s.
$$
\end{lemma}

\begin{lemma}\cite{amj}\label{moser}
\begin{enumerate}
\item If $f,\ g \in H^s \cap L^\infty$ and $|\zeta| \leq s$, then  there exists  a constant $C_s$ depending only on $s$ such that
\begin{equation}\label{moser1}
\begin{split}
\|\partial^\zeta_x (fg)\|_s\leq C_s(|f|_{\infty}|\nabla^s g|_2+|g|_{\infty}|\nabla^s f|_{2}).
\end{split}
\end{equation}

\item Let   $u(x)$ be a continuous function taking its values in some open set $G$ such that $u\in H^s \cap L^\infty$, and $g(u)$ be a smooth vector-valued function on  $G$. Then for $s\geq 1$,   there exists  a constant $C_s$ depending  only on $s$ such that
\begin{equation}\label{moser2}
\begin{split}
|\nabla^s g(u)|_2\leq C_s\Big \|\frac{\partial g}{\partial u }\Big\|_{s-1}|u|^{s-1}_{\infty}|\nabla^s u|_{2}.
\end{split}
\end{equation}
\end{enumerate}
\end{lemma}

The following lemma is a useful tool to improve weak convergence to strong convergence.
\begin{lemma}\cite{amj}\label{zheng5}
If the function sequence $\{w_n\}^\infty_{n=1}$ converges weakly  to $w$ in a Hilbert space $X$, then it converges strongly to $w$ in $X$ if and only if
$$
\|w\|_X \geq \lim \text{sup}_{n \rightarrow \infty} \|w_n\|_X.
$$
\end{lemma}


\section{Existence of regular solutions}

In this section, we aim at proving Theorem \ref{th2} and Corollary \ref{co2}. To this end, we first reformulated our main problem (\ref{eq:1.1})-(\ref{eq:2.211}) into a more convenient form. Since we will repeat the integration over $\mathbb{R}^3$ for many times, we will adopt to the following simplified notation
$$\int f=\int_{\mathbb{R}^3} f\text{d}x, $$
throughout this paper without further specification. All other integrals will be specified when they appear.

\subsection{Reformulation}
Setting $\phi=\rho^{\frac{\delta-1}{2}}$,   system (\ref{eq:1.1})  can be rewritten as
\begin{equation}
\begin{cases}
\label{eq:cccq}
\displaystyle
\phi_t+u\cdot \nabla \phi+\frac{\delta-1}{2}\phi \text{div} u=0,\\[10pt]
\displaystyle
u_t+u\cdot\nabla u +\frac{2A\gamma}{\delta-1}\phi^{\frac{2r-\delta-1}{\delta-1}}\nabla \phi+\phi^2 Lu=\nabla \phi^2 \cdot Q(u),
 \end{cases}
\end{equation}
where
\begin{equation} \label{eq:5.2qq}
\begin{split}
Q(u)=&\frac{\delta}{\delta-1}\big(\alpha(\nabla u+(\nabla u)^\top)+\beta \text{div}u \mathbb{I}_3\big)=\frac{\delta}{\delta-1}\mathbb{S}(u).
\end{split}
\end{equation}
The  initial data are given by
\begin{equation} \label{sfana1}
\begin{split}
&(\phi,u)|_{t=0}=(\phi_0,u_0)(x)=\left(\rho^{\frac{\delta-1}{2}}_0(x),u_0(x)\right),\quad x\in \mathbb{R}^3,
\end{split}
\end{equation}
with the far field behavior
\begin{equation}\label{sfanb1}
\begin{split}
(\phi,u)\rightarrow (\overline{\phi},0)=\left(\overline{\rho}^{\frac{\delta-1}{2}},0\right),\quad \text{as}\quad  |x|\rightarrow +\infty,\quad t>0.
\end{split}
\end{equation}

 To prove  Theorem \ref{th2}, we first establish the following existence result for the reformulated problem (\ref{eq:cccq})-(\ref{sfanb1}), and then in subsection $3.6$
 we will show that this result indeed implies Theorem \ref{th2}.
\begin{theorem}\label{ths1}
 If the initial data $( \phi_0, u_0)(x)$ satisfy
\begin{equation}\label{th78qq}
\begin{split}
&  \phi_0 \geq 0, \quad (\phi_0-\overline{\phi}, u_0)\in H^3,
\end{split}
\end{equation}
then there exists a positive time $T_*$ and a unique strong solution $( \phi, u)$ on $\mathbb{R}^3\times[0, T_*]$ to Cauchy problem (\ref{eq:cccq})-(\ref{sfanb1}), that is, $( \phi, u)$ is a solution of the Cauchy problem (\ref{eq:cccq})-(\ref{sfanb1}) in the sense of distribution and satisfies
\begin{equation}\label{reg11qq}\begin{split}
& \phi-\overline{\phi} \in C([0,T_*];H^3),\quad  \phi_t \in C([0,T_*];H^2),\\
& u\in C([0,T_*]; H^{s'})\cap L^\infty([0,T]; H^3), \quad \phi \nabla^4 u\in L^2([0,T_*] ; L^2), \\
& u_t \in C([0,T_*]; H^1)\cap L^2([0,T_*] ; D^2),
\end{split}
\end{equation}
for any constant $s' \in[2,3)$.
\end{theorem}

We will prove this theorem in the subsequent four subsections $3.2$-$3.5$ as explained in the introduction.

\subsection{Linearization with an artificial viscosity} In order to construct the local strong  solutions for the nonlinear problem (\ref{eq:cccq})-(\ref{sfanb1}), we first consider the following  linearized problem:
\begin{equation}\label{li4}
\begin{cases}
\displaystyle
\phi_t+v\cdot \nabla \phi+\frac{\delta-1}{2}\psi \text{div}v=0,\\[8pt]
\displaystyle
u_t+v\cdot\nabla u +\frac{A\gamma}{\gamma-1} \nabla \phi^{\frac{2\gamma-2}{\delta-1}}+(\phi^2+\eta^2) Lu=\nabla \phi^2 \cdot Q(v),\\[8pt]
(\phi,u)|_{t=0}=(\phi_0(x),u_0(x)),\quad x\in \mathbb{R}^3,\\[8pt]
\displaystyle
(\phi,u)\rightarrow (\overline{\phi},0),\quad \text{as}\quad  |x|\rightarrow +\infty,\quad t>0,
 \end{cases}
\end{equation}
where $\eta \in (0,1)$ is a  constant and
\begin{equation} \label{sseq:5.2qq}
\begin{split}
Q(v)=&\frac{\delta}{\delta-1}\big(\alpha(\nabla v+(\nabla v)^\top)+\beta \text{div}v \mathbb{I}_3\big)=\frac{\delta}{\delta-1}\mathbb{S}(v).
\end{split}
\end{equation}
Here $\psi$ and $v=\left(v^{(1)},v^{(2)},v^{(3)}\right)^\top\in \mathbb{R}^3$ are given functions satisfying the initial assumption $(\psi, v)(t=0,x)=(\phi_0, u_0)$ and the following regularities:
\begin{equation}\label{vg}
\begin{split}
& \psi-\overline{\phi}\in C([0,T]; H^3), \quad  \psi_t \in C([0,T]; H^2),\\
& v\in C([0,T] ; H^{s'})\cap L^\infty([0,T]; H^3) ,\quad   \psi \nabla^4 v\in L^2([0,T] ; L^2), \\
& v_t \in C([0,T] ; H^1)\cap L^2([0,T] ; D^2),
\end{split}
\end{equation}
for any constant $s'\in[2,3)$ and fixed time $T>0$.\\[2pt]

Such a linearization is specially designed to serve our purpose. Unlike
 \cite{sz3}, here we have to keep the positive and symmetric hyperbolic operator in the momentum equations. More precisely, we linearize
$$ u_t+u\cdot  \nabla u \rightarrow    u_t+v \cdot \nabla u,$$
instead of $ u_t+u\cdot  \nabla u \rightarrow    u_t+v \cdot \nabla v$ as in  \cite{sz3},
which plays a key role in the analysis on the regularities of velocity $u$.  Generally speaking, in the following proof of the existence, we will make a full use of   the transport property of the  hyperbolic part $ u_t+v \cdot \nabla u$.  On the other hand,  the strong diffusion term $(\phi^2+\eta^2) Lu$ helps to  ensure the global existence of the linearized problem. Moreover, we remark that there exists a carefully chosen compatibility  between the  last term $\frac{\delta-1}{2}\psi \text{div}v$ on the left hand side of $(\ref{li4})_1$ and the viscosity  term $(\phi^2+\eta^2) Lu$ in $(\ref{li4})_2$. Because of the lack of positive lower bound for $\phi_0$, it is not visible to show that there exists some constant $C$ independent of $\eta$ such that
$$
\|u\|_{L^2([0,T];D^4)}\leq C.
$$
Therefore, in order to obtain a finite upper bound of $|\phi|_{D^3}$ independent of $\eta$   from the continuity equation $ (\ref{li4}) _1$, we need a good estimate like
$\psi \nabla^4 v \in L^2([0,T];L^2)$.
For our solution $(\phi,u)$,
$  \|\phi \nabla^4 u\|_{L^2([0,T];L^2)}\leq C$
 can be obtained by the smoothing property of the artificial viscosity term $(\phi^2+\eta^2) Lu$. It is not
clear if other linearization approaches will achieve the same goal, this one \eqref{li4} succeeded.
The details can be found in Lemmas \ref{f2}-\ref{s5}.\\

As mentioned above, a rather standard method gives the following global existence of a strong solution $(\phi, u)$ to (\ref{li4})  for each fixed $\eta>0$.

 \begin{lemma}\label{lem1} For each fixed $\eta>0$ and any $T>0$, assume  that the initial data $(\phi_0,u_0)$ satisfy (\ref{th78qq}). Then there exists  a unique strong solution $(\phi,u)(x,t)$ in $\mathbb{R}^3 \times [0,T]$ to (\ref{li4}), that is, a solution satisfying the Cauchy problem (\ref{li4}) with the regularities
\begin{equation}\label{reggh}\begin{split}
& \phi-\overline{\phi} \in C([0,T]; H^3), \ \phi _t \in C([0,T]; H^2), \\
& u\in C([0,T]; H^{s'})\cap L^\infty([0,T]; H^3), \quad \phi \nabla^4 u\in L^2([0,T] ; L^2),  \\
& u_t \in C([0,T]; H^1)\cap L^2([0,T] ; D^2).
\end{split}
\end{equation}
\end{lemma}
\begin{proof} The first equation of (\ref{li4}) is a linear transport equation,
the existence and regularity of a unique solution $\phi$ in $\mathbb{R}^3 \times [0,T]$ to the first equation of (\ref{li4}) can be obtained by the standard theory of transport equation.

After knowing $\phi$ from $(\ref{li4})_1$, for $\eta>0$,   $(\ref{li4})_2$ becomes a  linear parabolic system of $u$. Then it is not difficult to solve $u$ by the standard parabolic theory.  Here we omit the details.
\end{proof}

In the following two subsections, we first  establish the uniform estimates independent of $\eta$, then we  pass to the limit $\eta\to 0$.

\subsection{Uniform a priori estimates  independent of $\eta$}
The main task of this subsection is to establish some local (in time) a priori estimates  for the solution  $(\phi, u)$ to (\ref{li4})  obtained in Lemma \ref{lem1}, independent of the artificial viscosity coefficient $\eta$. For this purpose, we fix a $T>0$ and choose a  positive constant $c_0$ large enough such  that
\begin{equation}\label{houmian}\begin{split}
2+\overline{\phi}+|\phi_0|_{\infty}+\|\phi_0-\overline{\phi}\|_{3}+\|u_0\|_{3}\leq c_0.
\end{split}
\end{equation}
We now assume that there exist some time $T^*\in (0,T)$ and constants $c_i$ ($i=1,2,3,4$) such that
$$1< c_0\leq c_1 \leq c_2 \leq c_3 \leq c_4,  $$
and
\begin{equation}\label{jizhu1}
\begin{split}
\sup_{0\leq t \leq T^*}\big(|\psi|^2_\infty+\| \psi(t)-\overline{\phi}\|^2_{3} +\| v(t)\|^2_{1}\big) \leq& c^2_1,\\
\sup_{0\leq t \leq T^*}\big(| \psi_t(t)|^2_{2} +|v(t)|^2_{D^2}+|v_t(t)|^2_{2}\big)+\int_{0}^{T^*} \Big( |\psi \nabla^3 v|^2_{2}+|v_t|^2_{D^1}\Big)\text{d}t \leq& c^2_2,\\
\text{ess}\sup_{0\leq t \leq T^*}\big(| \psi_t(t)|^2_{D^1}+|v(t)|^2_{D^3}+|v_t(t)|^2_{D^1}\big)+\int_{0}^{T^*} \Big( |\psi \nabla^4 v|^2_{2}+|v_t|^2_{D^2}\Big)\text{d}t
 \leq& c^2_3,\\
 \sup_{0\leq t \leq T^*}| \psi_t(t)|^2_{D^2} \leq& c^2_4,\\
\end{split}
\end{equation}
We shall determine $T^*$ and $c_i$ ($i=1,2,3,4$) later, see \eqref{dingyi} below,  so that they depend only on $c_0$ and the fixed constants $\overline{\rho}$, $\alpha$, $\beta$, $\gamma$, $A$, $\delta$ and $T$.

Let $(\phi,u)(x,t)$ be the unique strong solution to (\ref{li4}) in $\mathbb{R}^3 \times [0,T]$. We start from the estimates for $\phi$. Hereinafter, we use  $C\geq 1$ to denote  a generic positive constant depending only on fixed constants $\overline{\rho}$, $\alpha$, $\beta$,  $\gamma$, $A$, $\delta$ and $T$, which may vary each time when it appears.

\begin{lemma}\label{f2} Let $(\phi,u)(x,t)$ be the unique strong solution to (\ref{li4}) on $\mathbb{R}^3 \times [0,T]$. Then
\begin{equation*}\begin{split}
1+\overline{\phi}^2+|\phi(t)|^2_\infty+\|\phi(t)-\overline{\phi}\|^2_3\leq& Cc^2_0,\\
|\phi_t(t)|^2_2 \leq Cc^4_1,\quad |\phi_t(t)|^2_{D^1} \leq Cc^4_2,\quad |\phi_t(t)|^2_{D^2} \leq& Cc^4_3,
\end{split}
\end{equation*}
for $0\leq t \leq T_1=\min (T^*, (1+c_3)^{-2})$.
\end{lemma}

\begin{proof}Let $\tilde{\phi}=\phi-\overline{\phi}$. Then from $(\ref{li4})_1$, we have
\begin{equation}\label{bian}
\tilde{\phi}_t+v\cdot \nabla \tilde{\phi}+\frac{\delta-1}{2}\psi \text{div}v=0.
\end{equation}
Applying the operator $\partial^\zeta_x$  $(0\leq |\zeta|\leq 3)$ to $(\ref{bian})$, we obtain
\begin{equation}\label{guji1}
(\partial^\zeta_x \tilde{\phi})_t+v\cdot \nabla \partial^\zeta_x \tilde{\phi}=-(\partial^\zeta_x (v\cdot \nabla \tilde{\phi})-v\cdot \nabla \partial^\zeta_x \tilde{\phi})-\frac{\delta-1}{2} \partial^\zeta_x (\psi \text{div}v).
\end{equation}
Multiplying both sides of (\ref{guji1}) by $\partial^\zeta_x \tilde{\phi}$, and integrating over $\mathbb{R}^3$, we get
\begin{equation}\label{guji2}
\frac{1}{2}\frac{d}{dt}|\partial^\zeta_x \tilde{\phi}|^2_2\leq C\left(|\text{div}v|_\infty |\partial^\zeta_x \tilde{\phi}|^2_2+\Lambda^\zeta_1 |\partial^\zeta_x \tilde{\phi}|_2+\Lambda^\zeta_2 |\partial^\zeta_x \tilde{\phi}|_2\right),
\end{equation}
where
$$
\Lambda^\zeta_1=|\partial^\zeta_x (v\cdot \nabla \tilde{\phi})-v\cdot \nabla \partial^\zeta_x \tilde{\phi}|_2,\quad \Lambda^\zeta_2=|\partial^\zeta_x (\psi \text{div}v)|_2.
$$
From Lemma  \ref{zhen1},  letting
$r=b=2$, $a=+\infty$, $f=v$, $g=\nabla \tilde{\phi}$
 in (\ref{ku11}), we obtain
\begin{equation}\label{qe1}
\begin{split}
\Lambda^\zeta_1
\leq C_\zeta\left(|\nabla v|_\infty|\nabla^{|\zeta|}\tilde{\phi}|_2+|\nabla \tilde{\phi}|_\infty|\partial^\zeta_x v|_2\right).
\end{split}
\end{equation}

Now we  consider the term $\Lambda^\zeta_2$.  When $|\zeta|\leq 2$, it is easy to show that
\begin{equation}\label{qe2}
\Lambda^\zeta_2\leq C_\zeta\big(|\psi|_\infty+\|\nabla\psi\|_2\big) \|v\|_3;
\end{equation}
when $|\zeta|=3$, we have
\begin{equation}\label{qe4}
\begin{split}
\Lambda^\zeta_2\leq& C\left(|\psi \nabla^4 v|_2+|\nabla \psi\cdot\nabla^3 v|_2+|\nabla^2 \psi\cdot\nabla^2 v|_2+|\nabla^3 \psi\cdot\nabla v|_2\right)\\
\leq & C\left(|\psi \nabla^4 v|_2+\|\nabla\psi\|_2 \|v\|_3\right).
\end{split}
\end{equation}
Combining (\ref{guji2})-(\ref{qe4}), we arrive at
\begin{equation}\label{qe5}
\begin{split}
\frac{d}{dt} \|\tilde{\phi}(t)\|_3\leq& C\left(\|v\|_3\|\tilde{\phi}\|_3+ (|\psi|_\infty+\|\nabla\psi\|_2)\|v\|_3+|\psi \nabla^4 v|_2\right)\\
\leq& C\left(c_3\|\tilde{\phi}\|_3+c^2_3+|\psi \nabla^4 v|_2\right).
\end{split}
\end{equation}
From Gronwall's inequality, we  have
\begin{equation}\label{gb}\begin{split}
\|\tilde{\phi}(t)\|_3
\leq& \Big(\|\phi_0-\overline{\phi}\|_{3}+c^2_3t+\int_0^t |\psi \nabla^4 v|_2\text{d}s\Big) \exp (Cc_3t).
\end{split}
\end{equation}

Therefore, observing that
$$
\int_0^t | \psi \nabla^4 v(s)|_{2}\text{d}s\leq t^{\frac{1}{2}}\left(\int_0^t |\psi \nabla^4 v(s)|^2_{2}\text{d}s\right)^{\frac{1}{2}}\leq Cc_3t^{\frac{1}{2}},
$$
we get
$$\|\phi(t)-\overline{\phi}\|_3 \leq C c_{0} \quad \text{for}\quad 0\leq t \leq T_1=\min (T^{*}, (1+c_3)^{-2}).$$

The estimate for $\phi_t$ follows  from the relation:
$$\phi_t=-v\cdot \nabla \phi-\frac{\delta-1}{2}\psi\text{div} v.$$
For $0\leq t \leq T_1$, the following estimates hold
\begin{equation}\label{zhen6}
\begin{split}
|\phi_t(t)|_2\leq & C\big(|v(t)|_6|\nabla \phi(t)|_3+|\psi(t)|_\infty|\text{div} v(t)|_2\big)\\
\leq & Cc^2_1,\\
|\phi_t(t)|_{D^1}\leq & C\big(|v(t)|_\infty|\nabla^2 \phi(t)|_{2}+|\nabla v(t)|_6|\nabla \phi(t)|_3+|\psi(t)|_\infty\nabla^2 v(t)|_{2}\\
&\quad+|\nabla v(t)|_6|\nabla \psi(t)|_3\big)\\
\leq &Cc^2_2,\\
|\phi_t(t)|_{D^2}\leq & C\big(|v(t)|_\infty|\nabla^3 \phi(t)|_{2}+|\nabla v(t)|_\infty |\nabla^2 \phi(t)|_2+|\nabla^2 v(t)|_6|\nabla \phi(t)|_3\\
&\quad+|\psi(t)|_\infty|\nabla^3 v(t)|_{2}+|\nabla \psi(t)|_\infty|\nabla^2 v(t)|_2+|\nabla^2\psi(t)|_2|\nabla v(t)|_{\infty}\big)\\
\leq& Cc^2_3.
\end{split}
\end{equation}
These complete the  proof.

\end{proof}

Now we turn to the  estimates for  the velocity $u$ in the following two lemmas. In view of \eqref{100001}, we define
$$
K=  \frac{4\gamma-4}{\delta-1}\geq 8.
$$
The following lemma gives some lower order estimates for the velocity $u$.
 \begin{lemma}\label{4}Let $(\phi,u)(x,t)$ be the unique strong solution to (\ref{li4}) on $\mathbb{R}^3 \times [0,T]$. Then
\begin{equation*}
\begin{split}
\|u(t)\|^2_{ 1}\leq& Cc^{2}_0,\quad |u(t)|^2_{ D^2}+|u_t(t)|^2_{2}+\int_{0}^{t}\Big(|\phi \nabla^3 u(s)|^2_{2}+|u_t(s)|^2_{D^1}\Big)\text{d}s\leq Cc^{2K}_1
\end{split}
\end{equation*}
for $0 \leq t \leq T_2=\min(T^*,(1+c_3)^{-2K})$.
 \end{lemma}
\begin{proof} The proof will be carried out in two steps.\\

\noindent\underline{Step\ 1}. We estimate $|\partial^\zeta_x u|_2$ when $|\zeta|\leq 2$.  Taking  $\partial^\zeta_x$ on equation $(\ref{li4})_2$, we have
\begin{equation}\label{zhull1}
\begin{split}
&(\partial^\zeta_x u)_t+v\cdot \nabla (\partial^\zeta_x u)+(\phi^2+\eta^2) L\partial^\zeta_x u \\
=&-\frac{A\gamma}{\gamma-1}\partial^\zeta_x\nabla \phi^{\frac{2\gamma-2}{\delta-1}}+\nabla \phi^2 \cdot \partial^\zeta_x Q(v)- \left(\partial^\zeta_x(v\cdot \nabla u)-v\cdot \nabla (\partial^\zeta_x u)\right)\\
&- \Big(\partial^\zeta_x((\phi^2+\eta^2) Lu)-(\phi^2+\eta^2) L\partial^\zeta_x u \Big)+ \Big(\partial^\zeta_x(\nabla \phi^2  \cdot Q(v))-\nabla \phi^2  \cdot \partial^\zeta_x Q(v)\Big).
\end{split}
\end{equation}
Multiplying (\ref{zhull1}) by $\partial^\zeta_x u$ on both sides   and integrating  over $\mathbb{R}^3$ by parts,  we have
\begin{equation}\label{zhu101}
\begin{split}
&\frac{1}{2} \frac{d}{dt}|\partial^\zeta_x u|^2_2+\alpha| \sqrt{\phi^2+\eta^2} \nabla (\partial^\zeta_x u) |^2_2+(\alpha+\beta)|\sqrt{\phi^2+\eta^2} \text{div} \partial^\zeta_x u |^2_2\\[6pt]
\displaystyle
=& -\int (v\cdot \nabla (\partial^\zeta_x u))\cdot \partial^\zeta_x u-\frac{\delta-1}{\delta}\int \Big(\nabla (\phi^2+\eta^2) \cdot Q(\partial^\zeta_x u)\Big) \cdot \partial^\zeta_x u\\
&-\frac{A\gamma}{\gamma-1}\int \partial^\zeta_x\nabla \phi^{\frac{2\gamma-2}{\delta-1}}\cdot \partial^\zeta_x u+\int  \nabla \phi^2\cdot  \partial^\zeta_x Q(v) \cdot \partial^\zeta_x u\\
&
 -\int \Big(\partial^\zeta_x(v\cdot \nabla u)-v\cdot \nabla (\partial^\zeta_x u)\Big)\cdot \partial^\zeta_x u\\
&-\int \Big( \partial^\zeta_x ((\phi^2+\eta^2) Lu)-(\phi^2+\eta^2) L\partial^\zeta_x u \Big)\cdot \partial^\zeta_x u\\
&+\int  \Big(\partial^\zeta_x(\nabla \phi^2  \cdot Q(v))-\nabla \phi^2  \cdot \partial^\zeta_x Q(v)\Big)\cdot \partial^\zeta_x u\\
=:&\sum_{i=1}^{7} I_i.
\end{split}
\end{equation}
Now we estimate the last part of (\ref{zhu101}) term by term. Using  H\"older's inequality, Lemma \ref{lem2as} and Young's inequality, we have

\begin{equation}\label{zhu102}
\begin{split}
I_1=& -\int (v\cdot \nabla (\partial^\zeta_x u))\cdot \partial^\zeta_x u\\
\leq& C|\nabla v|_\infty |\partial^\zeta_x u|^2_2\leq Cc_3 |\partial^\zeta_x u|^2_2,\\
I_2=&-\frac{\delta-1}{\delta}\int \Big(\nabla (\phi^2+\eta^2) \cdot Q(\partial^\zeta_x u)\Big) \cdot \partial^\zeta_x u\\
\leq &C|\phi \nabla (\partial^\zeta_x u) |_2|\nabla \phi|_\infty |\partial^\zeta_x u |_2\\
\leq &  C|\sqrt{\phi^2+\eta^2} \nabla (\partial^\zeta_x u) |_2|\nabla \phi |_\infty|\partial^\zeta_x u |_2\\
\leq & \frac{\alpha}{20}|\sqrt{\phi^2+\eta^2} \nabla (\partial^\zeta_x u) |^2_2+Cc^2_0 |\partial^\zeta_x u|^2_2.
\end{split}
\end{equation}

For the term $I_3$, it is obvious that
\begin{equation}\label{zhu1022}
\begin{split}
I_3=&-\frac{A\gamma}{\gamma-1}\int \nabla \partial^\zeta_x\phi^{\frac{2\gamma-2}{\delta-1}}\cdot \partial^\zeta_x u
=\frac{A\gamma}{\gamma-1}\int  \partial^\zeta_x \phi^{\frac{2\gamma-2}{\delta-1}}\text{div} \partial^\zeta_x u.
\end{split}
\end{equation}
Since  $1< \delta \leq \frac{\gamma+1}{2}$,
$\frac{2\gamma-2}{\delta-1}\geq 4$.
Then for $|\zeta|=0$, one has
\begin{equation}\label{zhu1022as}
\begin{split}
I_3=&-\frac{2A\gamma}{\delta-1}\int \phi^{\frac{2\gamma-\delta-1}{\delta-1}}\nabla \phi\cdot u\\
\leq& C|\phi|^{\frac{2\gamma-\delta-1}{\delta-1}}_\infty|\nabla \phi|_2| u|_2\\
\leq & C|\phi|^{\frac{4\gamma-2\delta-2}{\delta-1}}_\infty|\nabla \phi|^2_2+C| u|^2_2\\
\leq & Cc^{\frac{4\gamma-4}{\delta-1}}_0+C| u|^2_2.
\end{split}
\end{equation}
When $1\leq |\zeta|\leq 2$, we have
\begin{equation}\label{zhu103}
\begin{split}
I_3=&\frac{A\gamma}{\gamma-1}\int  \partial^\zeta_x \phi^{\frac{2\gamma-2}{\delta-1}}\text{div} \partial^\zeta_x u\\
\leq& C\Big(|\phi|^{\frac{2\gamma-2}{\delta-1}-2}_\infty|\nabla \phi|_2|\phi \text{div}\partial^\zeta_x u|_2+|\phi|^{\frac{2\gamma-2}{\delta-1}-2}_\infty|\nabla^2 \phi|_2|\phi \text{div}\partial^\zeta_x u|_2\\
&\quad+|\phi|^{\frac{2\gamma-2}{\delta-1}-3}_\infty|\nabla \phi|_6|\nabla \phi|_3|\phi \text{div}\partial^\zeta_x u|_2\Big)\\
\leq& C\Big(|\phi|^{\frac{2\gamma-2}{\delta-1}-2}_\infty\|\nabla \phi\|_2+|\phi|^{\frac{2\gamma-2}{\delta-1}-3}_\infty|\nabla \phi|_6|\nabla \phi|_3\Big) |\phi\nabla (\partial^\zeta_x u)|_2\\
\leq& Cc^{\frac{2\gamma-\delta-1}{\delta-1}}_0|\phi\nabla (\partial^\zeta_x u)|_2\\
\leq& Cc^{\frac{4\gamma-2\delta-2}{\delta-1}}_0+\frac{\alpha}{20}|\sqrt{\phi^2+\eta^2}\nabla (\partial^\zeta_x u)|^2_2,
\end{split}
\end{equation}

which, combining with (\ref{zhu1022as}) and $\delta >1$, implies that
\begin{equation}\label{zhu1055}
\begin{split}
I_3\leq \frac{\alpha}{10}|\sqrt{\phi^2+\eta^2} \nabla (\partial^\zeta_x u) |^2_2+C|u|^2_2+Cc^{\frac{4\gamma-4}{\delta-1}}_0,\quad \text{for}\quad  |\zeta|\leq 2.
\end{split}
\end{equation}

For the term $I_4$, it is not difficult to show that
\begin{equation}\label{zhu106}
\begin{split}
I_4=&\int  \nabla \phi^2\cdot  \partial^\zeta_x Q(v) \cdot \partial^\zeta_x u\\
\leq & C|\phi|_\infty|\nabla \phi|_\infty \|\nabla v\|_2 | \partial^\zeta_x u|_2
\leq Cc^3_3 |\partial^\zeta_x u |_2,\quad \text{for}\quad   |\zeta|\leq  2.
\end{split}
\end{equation}

For the term $I_5$, we have
\begin{equation}\label{zhu10606}
\begin{split}
I_5=&-\int \Big(\partial^\zeta_x(v\cdot \nabla u)-v\cdot \nabla (\partial^\zeta_x u)\Big)\cdot \partial^\zeta_x u\\
\leq & | \partial^\zeta_x(v\cdot \nabla u)-v\cdot \nabla (\partial^\zeta_x u)|_2|\partial^\zeta_x u|_2\\
\leq &C(|\nabla v|_\infty \|\nabla  u\|_1+|\partial^\zeta_x v|_6|\nabla u|_3)|\partial^\zeta_x u|_2\\
\leq & Cc_3\| u\|^2_2+Cc_3|\nabla u|_3|\partial^\zeta_x u|_2\leq Cc_3\| u\|^2_2, \quad \text{for}\quad  1\leq |\zeta|\leq  2.
\end{split}
\end{equation}

For the term $I_6$,  when $ |\zeta|= 1$ we have,
\begin{equation}\label{zhu107}
\begin{split}
I_6=&-\int  \Big(\partial^\zeta_x((\phi^2+\eta^2) Lu)-(\phi^2+\eta^2) L\partial^\zeta_x u \Big)\cdot \partial^\zeta_x u\\
\leq &C|\partial^\zeta_x((\phi^2+\eta^2) Lu)-(\phi^2+\eta^2) L\partial^\zeta_x u |_2 |\partial^\zeta_x u|_2\\
\leq &C|\nabla (\phi^2+\eta^2)|_\infty |L u|_2 |\partial^\zeta_x u|_2\\
\leq& C|\phi|_\infty |\nabla \phi|_\infty|u|_{D^1}|u|_{D^2}
\leq Cc^2_0\|\nabla u\|^2_{1}.
\end{split}
\end{equation}
When $|\zeta|=2$, we have
\begin{equation}\label{zhu109}
\begin{split}
I_6=&-\int  \Big(\partial^\zeta_x((\phi^2+\eta^2) Lu)-(\phi^2+\eta^2) L\partial^\zeta_x u\Big)\cdot \partial^\zeta_x u\\
\leq &C\int  \Big(|\nabla^2\phi||\phi Lu|+|\nabla \phi|^2|Lu|+|\phi \nabla Lu||\nabla \phi|\Big)|\partial^\zeta_x u|\\
\leq &C\Big(  |\nabla ^2 \phi|_3 |\phi\nabla ^2 u|_6 |\nabla ^2 u|_2+|\nabla \phi|^2_\infty|u|^2_{D^2}+|\nabla \phi|_\infty|\phi\nabla ^3 u|_2 |\nabla ^2 u|_2\Big)\\
\leq&  C\|\nabla \phi\|^2_2 |u|^2_{D^2}+\frac{\alpha}{20} |\phi \nabla^3 u|^2_2\\
\leq & Cc^2_0|u|^2_{D^2}+\frac{\alpha}{20} |\sqrt{\phi^2+\eta^2} \nabla^3 u|^2_2,
\end{split}
\end{equation}
where we used the fact that
\begin{equation}\label{quandeng}
|\phi\nabla ^2 u|_6\leq C|\phi \nabla ^2 u|_{D^1}\leq C(|\nabla \phi|_\infty|\nabla^2 u|_2+|\phi \nabla^3 u|_2).
\end{equation}
Combining (\ref{zhu107})-(\ref{zhu109}), we then have
\begin{equation}\label{zhu10111}
\begin{split}
I_6
\leq &\frac{\alpha}{20}|\sqrt{\phi^2+\eta^2}\nabla^3 u|^2_2+ Cc^2_0 \|\nabla u\|^2_{1}, \quad \text{for}\quad  1\leq |\zeta|\leq  2.
\end{split}
\end{equation}

Term $I_7$ can be controlled as follows,
\begin{equation}\label{zhu1012}
\begin{split}
I_7=&\int  \Big(\partial^\zeta_x(\nabla \phi^2  \cdot Q(v))-\nabla \phi^2 \cdot \partial^\zeta_x Q(v)\Big)\cdot \partial^\zeta_x u\\
\leq &|\partial^\zeta_x(\nabla \phi^2  Q(v))-\nabla \phi^2 \partial^\zeta_x Q(v)|_2|\partial^\zeta_x u|_2.
\end{split}
\end{equation}
When $|\zeta|=1$,
\begin{equation}\label{zhu1013}
\begin{split}
&|\partial^\zeta_x(\nabla \phi^2  \cdot Q(v))-\nabla \phi^2 \cdot \partial^\zeta_x Q(v)|_2\\
\leq &C(|\nabla \phi|^2_\infty |\nabla  v |_2+|\phi|_\infty|\nabla^2 \phi|_2 |\nabla v|_\infty\big)\leq Cc^3_3;
\end{split}
\end{equation}
when $|\zeta|=2$,
\begin{equation}\label{zhu1013rt}
\begin{split}
&|\partial^\zeta_x(\nabla \phi^2  \cdot Q(v))-\nabla \phi^2 \cdot \partial^\zeta_x Q(v)|_2\\
\leq &C\Big(|\nabla \phi|^2_\infty |\nabla^2 v |_2+|\phi|_\infty|\nabla^2 \phi|_6 |\nabla^2 v|_3\\
&+|\phi|_\infty |\nabla^3 \phi |_2|\nabla v|_\infty+|\nabla \phi|_\infty|\nabla^2 \phi|_2 |\nabla v|_\infty\Big)\\
\leq& Cc^3_3,
\end{split}
\end{equation}
which, combining with (\ref{zhu1012}), implies that
\begin{equation}\label{zhu1014}
\begin{split}
I_7
\leq& Cc^3_3|\partial^\zeta_x u|_2, \quad \text{for}\quad  1\leq |\zeta|\leq  2.
\end{split}
\end{equation}

Therefore, from  (\ref{zhu101})-(\ref{zhu102}), (\ref{zhu1055})-(\ref{zhu10606}), (\ref{zhu10111}) and (\ref{zhu1014}), we deduce that
\begin{equation}\label{zhu1021}
\begin{split}
&\frac{d}{dt}\|u\|^2_2+|\sqrt{\phi^2+\eta^2} \nabla^3 u|^2_2
\leq  Cc^{2}_3 \|u\|^2_{2}+Cc^{K}_3.
\end{split}
\end{equation}

Now Gronwall's inequality implies that
\begin{equation}\label{zhu10222-1}
\begin{split}
\|u(t)\|^2_2+\int_0^t |\sqrt{\phi^2+\eta^2} \nabla^3 u|^2_2\text{d}s  \leq  \big(Cc^{K}_3t+\|u_0\|^2_2\big)\exp (Cc^{2}_3 t)\leq Cc^2_0,
\end{split}
\end{equation}
for $0\leq t\leq T_2=\min \big(T^*,(1+c_3)^{-2K})$.\\

\noindent\underline{Step 2}. We estimate $|\partial^\zeta_x u_t|_2$ when $|\zeta|\leq 1$.
From the momentum equations $(\ref{li4})_2$, we  have
\begin{equation}\label{zhu1019ss}
\begin{split}
|u_t|_2=&\Big|v\cdot\nabla u+\frac{A\gamma}{\gamma-1}\nabla \phi^{\frac{2\gamma-2}{\delta-1}}+(\phi^2+\eta^2) Lu-\nabla \phi^2 \cdot  Q(v)\Big|_2\\
\leq &C\Big(|v|_6 |\nabla u|_3+|\phi|^{\frac{2\gamma-2}{\delta-1}-1}_\infty|\nabla\phi|_2 +|\phi^2+\eta^2|_\infty|u|_{D^2}+|\phi|_\infty|\nabla\phi|_\infty|\nabla v|_2\Big)\\
\leq& Cc^{K}_1.
\end{split}
\end{equation}
Similarly, for $|u_t|_{D^1}$, we  have
\begin{equation}\label{zhu1020ss}
\begin{split}
|u_t|_{D^1}=& \Big|v\cdot\nabla u+\frac{A\gamma}{\gamma-1}\nabla \phi^{\frac{2\gamma-2}{\delta-1}}+(\phi^2+\eta^2) Lu-\nabla \phi^2 \cdot Q(v)\Big|_{D^1}\\
\leq & C\Big(\|v\|_2\|\nabla u\|_1+|\phi|^{\frac{2\gamma-2}{\delta-1}-1}_\infty|\nabla^2\phi|_2+|\phi|^{\frac{2\gamma-2}{\delta-1}-2}_\infty|\nabla\phi|_2|\nabla\phi|_\infty\\
&\quad+|\sqrt{\phi^2+\eta^2}|_\infty|\sqrt{\phi^2+\eta^2} \nabla^3 u|_2+|\phi|_\infty|\nabla \phi|_\infty|u|_{D^2}\\
&\quad+(|\phi|^2_\infty+\|\nabla\phi\|^2_2)\|\nabla v\|_1\Big)\\
\leq & C\Big(c^{K}_2+c_0|\sqrt{\phi^2+\eta^2} \nabla^3 u|_2\Big),
\end{split}
\end{equation}
which  implies that
$$
\int_{0}^t |u_t|^2_{D^1} \text{d}s \leq C\int_0^t \big(c^{2K}_2+c^2_0|\sqrt{\phi^2+\eta^2} \nabla^3 u|^2_2\big) \text{d}s\leq Cc^{4}_0,
$$
for $\ 0\leq t \leq T_2=\min(T^*,(1+c_3)^{-2K})$.
\end{proof}

Finally, we  give the higher order estimates for  velocity $u$ such as $|u|_{L^\infty([0,T];D^3)}$ and $|\phi \nabla^4 u|_{L^2([0,T]; L^2)}$ through  an accurate analysis on the artificial viscosity $(\phi^2+\eta^2)Lu$.
 \begin{lemma}\label{s5}Let $(\phi,u)(x,t)$ be the unique strong solution to (\ref{li4}) on $\mathbb{R}^3 \times [0,T]$. Then
\begin{equation*}
\begin{split}
|u(t)|^2_{ D^3}+|u_t(t)|^2_{D^1}+\int_{0}^{t}\Big(|\phi \nabla^4 u(s)|^2_{2}+|u_t(s)|^2_{D^2}\Big)\text{d}s\leq Cc^{2K}_2
\end{split}
\end{equation*}
for $0 \leq t \leq T_3=\min(T^*,(1+c_3)^{-2K-4})$.
 \end{lemma}
\begin{proof}The proof is divided into two steps.\\

\noindent\underline{Step\ 1}. The estimate of $|\partial^\zeta_x u|_2$ for $|\zeta|=3$.
Multiplying (\ref{zhull1}) by $\partial^\zeta_x u$ on both sides   and integrating  over $\mathbb{R}^3$ by parts,  we get
\begin{equation}\label{szhu101}
\begin{split}
&\frac{1}{2} \frac{d}{dt}|\partial^\zeta_x u|^2_2+\alpha| \sqrt{\phi^2+\eta^2} \nabla (\partial^\zeta_x u) |^2_2+(\alpha+\beta)|\sqrt{\phi^2+\eta^2} \text{div} \partial^\zeta_x u |^2_2\\[6pt]
\displaystyle
=& -\int (v\cdot \nabla (\partial^\zeta_x u))\cdot \partial^\zeta_x u-\frac{\delta-1}{\delta}\int (\nabla (\phi^2+\eta^2) \cdot Q(\partial^\zeta_x u)) \cdot \partial^\zeta_x u\\
&-\frac{A\gamma}{\gamma-1}\int \partial^\zeta_x\nabla \phi^{\frac{2\gamma-2}{\delta-1}}\cdot \partial^\zeta_x u+\int  \nabla \phi^2\cdot  \partial^\zeta_x Q(v) \cdot \partial^\zeta_x u\\
&
 -\int \Big(\partial^\zeta_x(v\cdot \nabla u)-v\cdot \nabla (\partial^\zeta_x u)\Big)\cdot \partial^\zeta_x u\\
&-\int\Big(\partial^\zeta_x ((\phi^2+\eta^2) Lu)-(\phi^2+\eta^2) L\partial^\zeta_x u \Big)\cdot \partial^\zeta_x u\\
&+\int  \Big(\partial^\zeta_x(\nabla \phi^2  \cdot Q(v))-\nabla \phi^2\cdot \partial^\zeta_x Q(v)\Big)\cdot \partial^\zeta_x u\\
=:&\sum_{i=8}^{14} I_i,
\end{split}
\end{equation}

Now we estimate $I_i$ $(i=8,\cdots, 14)$ term by term. Similarly to the derivation of (\ref{zhu102}), the first
two terms can be treated as follows,
\begin{equation}\label{szhu102}
\begin{split}
I_8=& -\int (v\cdot \nabla (\partial^\zeta_x u))\cdot \partial^\zeta_x u \leq Cc_3 |\partial^\zeta_x u|^2_2,\\
 I_{9}=&-\frac{\delta-1}{\delta}\int (\nabla (\phi^2+\eta^2) \cdot Q(\partial^\zeta_x u)) \cdot \partial^\zeta_x u\leq   \frac{\alpha}{20}|\sqrt{\phi^2+\eta^2} \nabla (\partial^\zeta_x u) |^2_2+Cc^2_0|\partial^\zeta_x u|^2_2.
\end{split}
\end{equation}

For the term $I_{10}$, when $|\zeta|=3$, one has
\begin{equation}\label{szhu1022}
\begin{split}
I_{10}=&-\frac{A\gamma}{\gamma-1}\int \nabla \partial^\zeta_x\phi^{\frac{2\gamma-2}{\delta-1}}\cdot \partial^\zeta_x u
=\frac{A\gamma}{\gamma-1}\int  \partial^\zeta_x \phi^{\frac{2\gamma-2}{\delta-1}}\text{div} \partial^\zeta_x u\\
\leq& C\int \Big(|\phi|^{\frac{2\gamma-2}{\delta-1}-4}|\nabla \phi|^3+|\phi|^{\frac{2\gamma-2}{\delta-1}-3} |\nabla \phi||\nabla^2\phi|+|\phi|^{\frac{2\gamma-2}{\delta-1}-2}|\nabla^3 \phi|\Big)|\phi \text{div} \partial^\zeta_x u|\\
\leq & C\Big(|\nabla \phi|_2 |\nabla \phi|^2_\infty |\phi|^{\frac{2\gamma-2}{\delta-1}-4}_\infty|\phi \nabla (\partial^\zeta_x u)|_2+|\nabla^2 \phi|_2 |\nabla \phi|_\infty |\phi|^{\frac{2\gamma-2}{\delta-1}-3}_\infty|\phi \nabla (\partial^\zeta_x u)|_2\\
&\quad+|\nabla^3 \phi|_2 |\phi|^{\frac{2\gamma-2}{\delta-1}-2}_\infty|\phi \nabla (\partial^\zeta_x u)|_2\Big)\\
\leq &C\Big(|\nabla \phi|^2_2 |\nabla \phi|^4_\infty |\phi|^{\frac{4\gamma-4}{\delta-1}-8}_\infty+|\nabla^2 \phi|^2_2 |\nabla \phi|^2_\infty |\phi|^{\frac{4\gamma-4}{\delta-1}-6}_\infty+|\nabla^3 \phi|^2_2 |\phi|^{\frac{4\gamma-4}{\delta-1}-4}_\infty\Big)\\
&\quad +\frac{\alpha}{20}|\phi \nabla (\partial^\zeta_x u)|^2_2\\
\leq& Cc^{\frac{4\gamma-2\delta-2}{\delta-1}}_0+\frac{\alpha}{20}|\sqrt{\phi^2+\eta^2}\nabla (\partial^\zeta_x u)|^2_2.
\end{split}
\end{equation}

For the term $I_{11}$, from integration by parts, we  have
\begin{equation}\label{szhu106}
\begin{split}
I_{11}=&\int  \nabla \phi^2\cdot  \partial^\zeta_x Q(v) \cdot \partial^\zeta_x u\\
\leq& C\int \Big(|\nabla^2 \phi| |\nabla^3 v| |\phi \partial^\zeta_x u|+|\nabla \phi|^2 |\nabla^3 v| | \partial^\zeta_x u|+|\nabla \phi| |\nabla^3 v| |\phi \nabla \partial^\zeta_x u|\Big)\\
\leq&  C\Big(|\nabla^2 \phi|_3 |\nabla^3  v|_2 |\phi \partial^\zeta_x u|_6+|\nabla \phi|^2_\infty |\nabla^3  v|_2 |\partial^\zeta_x u|_2+|\nabla \phi|_\infty |\nabla^3  v|_2 |\phi \nabla (\partial^\zeta_xu) |_2\Big)\\
\leq &C\Big(|\nabla^2 \phi|_3 |\nabla^3  v|_2\big(|\phi \nabla (\partial^\zeta_x u) |_2+|\nabla \phi|_\infty| \partial^\zeta_x u |_2\big)+c^3_3 |\partial^\zeta_x u|_2+c^2_3 |\phi \nabla(\partial^\zeta_xu) |_2\Big)\\
\leq & C\Big(c^3_3 |\partial^\zeta_x u |_2+c^4_3\Big) +\frac{\alpha}{20}|\sqrt{\phi^2+\eta^2} \nabla (\partial^\zeta_x u) |^2_2,
\end{split}
\end{equation}
where we used the fact that
\begin{equation}\label{squandeng}
|\phi\nabla ^3 u|_6\leq C|\phi \nabla ^3 u|_{D^1}\leq C\big(|\nabla \phi|_\infty|\nabla^3 u|_2+|\phi \nabla^4 u|_2\big).
\end{equation}

For the term  $I_{12}$,  letting
$r=b=2$, $a=+\infty$, $f=v$, $g=\nabla u$
 in (\ref{ku11}) of Lemma \ref{zhen1}, from (\ref{zhu10222-1}) we obtain
\begin{equation}\label{szhu10606}
\begin{split}
I_{12}=&-\int \Big(\partial^\zeta_x(v\cdot \nabla u)-v\cdot \nabla (\partial^\zeta_x u)\Big)\cdot \partial^\zeta_x u\\
\leq & | \partial^\zeta_x(v\cdot \nabla u)-v\cdot \nabla (\partial^\zeta_x u)|_2|\partial^\zeta_x u|_2\\
\leq &C\left(|\nabla v|_\infty |\nabla^3 u|_2+|\nabla^3 v|_2|\nabla u|_\infty\right)|\partial^\zeta_x u|_2\\
\leq&  C\left(c_3|\nabla^3 u|^2_2+c_3\|\nabla u\|_2|\nabla^3 u|_2\right)\\
\leq & C\left(c_3|\nabla^3 u|^2_2+c_3(c_0+|\nabla^3 u|_2)|\nabla^3 u|_2 \right)\\
\leq & C\left(c^2_3|\nabla^3 u|^2_2+c^{2}_3\right).
\end{split}
\end{equation}

For the term $I_{13}$, from $|\zeta|=3$ and (\ref{squandeng}) we have
\begin{equation}\label{szhu1011}
\begin{split}
I_{13}=&-\int  \big(\partial^\zeta_x((\phi^2+\eta^2) Lu)-(\phi^2+\eta^2) L\partial^\zeta_x u\big)\cdot \partial^\zeta_x u\\
\leq &C\int \Big(|\nabla^3 \phi| |Lu| |\phi\partial^\zeta_x u|+|\nabla \phi| |\nabla^2 \phi| |Lu| |\partial^\zeta_x u|+|\nabla \phi|^2 |\nabla Lu| |\partial^\zeta_x u|\Big)\\
&+C\int \Big(|\nabla^2\phi| |\phi \nabla Lu| |\partial^\zeta_x u|+|\phi \nabla^2 Lu| |\nabla \phi| |\partial^\zeta_x u| \Big)\\
\leq &C\Big(|\nabla^3 \phi|_2|\nabla^2u|_{3}|\phi\nabla ^3 u|_6 +|\nabla \phi|_\infty |\nabla ^2 \phi|_3 |\nabla ^2 u|_6 | \nabla ^3 u|_2+|\nabla \phi|^2_\infty   | \nabla ^3 u|^2_2\\[2pt]
&\quad+|\nabla^2 \phi|_3|\nabla^3u|_{2}|\phi\nabla ^3 u|_6+|\nabla \phi|_\infty  |\phi \nabla ^4 u|_2 | \nabla ^3 u|_2\Big)\\[2pt]
\leq & C\Big(c_0| \nabla ^2 u|^{\frac{1}{2}}_2| \nabla ^3 u|^{\frac{1}{2}}_2\big(|\phi\nabla^4 u|_2+|\nabla \phi|_\infty |\nabla^3 u|_2\big)\\[2pt]
&\quad+ c_0|u|_{D^3}\big(|\phi\nabla^4 u|_2+|\nabla \phi|_\infty| \nabla^3 u|_2\big)+ c^2_0 |u|^2_{D^3}\Big)+ \frac{\alpha}{20}|\phi \nabla^4 u|^2_2\\
\leq&C\left(c^{4}_3 |u|^2_{D^3}+c^{2K+4}_3 \right)+\frac{\alpha}{20}|\sqrt{\phi^2+\eta^2}\nabla^4 u|^2_2.
\end{split}
\end{equation}

For the term $I_{14}$, we notice that
\begin{equation*}
\begin{split}
&\partial^\zeta_x(\nabla \phi^2  \cdot Q(v))-\nabla \phi^2 \cdot \partial^\zeta_x Q(v)\\
=&\sum_{1\le i,j,k\le 3}l_{ijk}\Big(C_{1ijk} \nabla \partial^{\zeta^i}_x \phi^2\cdot  \partial^{\zeta^j+\zeta^k}_x Q(v) +C_{2ijk} \nabla \partial^{\zeta^j+\zeta^k}_x \phi^2 \cdot \partial^{\zeta^i}_x Q(v)\Big)+ \nabla \partial^\zeta_x \phi^2 \cdot Q(v),
\end{split}
\end{equation*}
 where $\zeta=\zeta^1+\zeta^2+\zeta^3$ for three multi-indexes $\zeta^i\in \mathbb{R}^3$ $(i=1,2,3)$ satisfying $|\zeta^i|=1$; $C_{1ijk}$ and  $C_{2ijk}$  are all constants. The summation is through all permutation on $i, j, k$. $l_{ijk}=1$ if $i$, $j$ and $k$ are different from each other,  and otherwise $l_{ijk}=0$ since there is no duplicated differentiation with respect to such decomposition on $\zeta$.
Then, we deduce that
\begin{equation}\label{szhu1015}
\begin{split}
I_{14}=&\int  \big(\partial^\zeta_x(\nabla \phi^2 \cdot  Q(v))-\nabla \phi^2 \cdot \partial^\zeta_x Q(v)\big)\cdot \partial^\zeta_x u\\
=&\int \Big( \sum_{i,j,k}l_{ijk}C_{1ijk} \nabla \partial^{\zeta^i}_x \phi^2\cdot \partial^{\zeta^j+\zeta^k}_x  Q(v)\Big)\cdot \partial^\zeta_x u\\
&+\int \Big(\sum_{i,j,k}l_{ijk}C_{2ijk} \nabla \partial^{\zeta^j+\zeta^k}_x \phi^2\cdot  \partial^{\zeta^i}_x  Q(v)  \Big)\cdot \partial^\zeta_x u
+\int \big( \nabla \partial^\zeta_x \phi^2 \cdot Q(v)  \big)\cdot \partial^\zeta_x u\\
=:&I_{141}+I_{142}+I_{143}.
\end{split}
\end{equation}
It follows that
\begin{equation}\label{zhu1018}
\begin{split}
I_{141}=&\int \Big( \sum_{i,j,k}l_{ijk}C_{1ijk} \nabla \partial^{\zeta^i}_x \phi^2 \partial^{\zeta^j+\zeta^k}_x Q(v)\Big)\cdot \partial^\zeta_x u\\
\leq &C\left(|\nabla \phi|^2_\infty|\nabla^3 u|_2|\nabla^3 v|_2+|\nabla^3 v|_2|\phi\nabla^3 u|_6 |\nabla^2 \phi|_3\right)\\
\leq & C\left(c^3_3|\nabla^3 u|_2+c^2_3\big(|\sqrt{\phi^2+\eta^2} \nabla^4 u|_2+|\nabla \phi|_\infty| \nabla^3 u|_2\big)\right)\\
\leq &  C\left(c^{3}_3 |u|_{D^3}+c^{4}_3\right)+\frac{\alpha}{20}|\sqrt{\phi^2+\eta^2} \nabla^4 u|^2_2,\\
I_{142}=&\int \Big(\sum_{i,j,k}l_{ijk}C_{2ijk} \nabla \partial^{\zeta^j+\zeta^k}_x \phi^2 \cdot \partial^{\zeta^i}_x Q(v) \Big)\cdot \partial^\zeta_x u\\
\leq& C\left(|\nabla^3 \phi|_2|\nabla^2 v|_3|\phi\nabla^3 u|_6+|\nabla \phi|_\infty|\nabla^2 \phi|_3|\nabla^2 v|_6|\nabla^3 u|_2\right)\\
\leq &  C\left(c^3_3|\nabla^3 u|_2+c^2_3\big(|\sqrt{\phi^2+\eta^2} \nabla^4 u|_2+|\nabla \phi|_\infty |\nabla^3 u|_2\big)\right)\\
\leq & C\left(c^{3}_3 |u|_{D^3}+c^{4}_3\right)+\frac{\alpha}{20}|\sqrt{\phi^2+\eta^2} \nabla^4 u|^2_2.
\end{split}
\end{equation}
For the  term $I_{143}$, from integration by parts, we have
\begin{equation}\label{szhu1016}
\begin{split}
I_{143}
=&\int \Big( \nabla \partial^\zeta_x \phi^2 Q(v)\Big) \cdot \partial^\zeta_x u\\
=&-\int \sum_{i=1}^3 \Big( \partial^{\zeta-\zeta^i}_x \nabla \phi^2 \cdot \partial^{\zeta^i}_x Q(v) \cdot \partial^\zeta_x u\Big)-\int \sum_{i=1}^3 \Big( \partial^{\zeta-\zeta^i}_x \nabla \phi^2 \cdot  Q(v) \cdot  \partial^{\zeta+\zeta^i}_x u\Big)\\
=:&I_{1431}+I_{1432}.
\end{split}
\end{equation}
For simplicity, we only consider the case that $i=1$,  the rest terms can be estimated similarly. When $i=1$, the corresponding term in $I_{1431}$ is
\begin{equation}\label{szhu1016}
\begin{split}
I^{(1)}_{1431}
=&-\int  \partial^{\zeta^2+\zeta^3}_x \nabla \phi^2 \cdot \partial^{\zeta^1}_x Q(v) \cdot \partial^\zeta_x u\\
\leq & C\left(|\nabla^3 \phi|_2|\nabla^2 v|_3|\phi\nabla^3 u|_6+|\nabla \phi|_\infty|\nabla^2 \phi|_6|\nabla^2 v|_3|\nabla^3 u|_2\right)\\
\leq & C\left(c^3_3 |u|_{D^3}
+ c^2_3\big(|\phi\nabla^4 u|_2+|\nabla \phi \nabla^3 u|_2\big)\right)\\
\leq &C\left(c^3_3 |u|_{D^3}+c^4_3\right)+\frac{\alpha}{20}|\sqrt{\phi^2+\eta^2}\nabla^4 u|^2_2,
\end{split}
\end{equation}
and the corresponding term in  $I_{1432}$ is
\begin{equation}\label{zhu1016kl}
\begin{split}
I^{(1)}_{1432}=&-\int  \partial^{\zeta^2+\zeta^3}_x \nabla \phi^2 \cdot  Q(v) \cdot  \partial^{\zeta+\zeta^1}_x u\\
=&-2\int \Big(\partial^{\zeta^2+\zeta^3}_x \nabla \phi \phi+\partial^{\zeta^2}_x\nabla \phi \partial^{\zeta^3}_x\phi)\Big) \cdot  Q(v) \cdot  \partial^{\zeta+\zeta^1}_x u\\
&-2\int \Big(\partial^{\zeta^3}_x \nabla \phi \partial^{\zeta^2}_x\phi+ \nabla \phi \partial^{\zeta^2+\zeta^3}_x\phi\Big) \cdot  Q(v) \cdot  \partial^{\zeta+\zeta^1}_x u\\
=:&I_A+I_B+I_C+I_D.
\end{split}
\end{equation}
It is not hard to show that
\begin{equation}\label{szhu1016kl}
\begin{split}
I_A=&-2\int \Big(\partial^{\zeta^2+\zeta^3}_x \nabla \phi \phi\Big) \cdot  Q(v) \cdot  \partial^{\zeta+\zeta^1}_xu\\
\leq &C|\nabla^3 \phi|_2|\nabla v|_\infty|\phi\nabla^4 u|_2\leq Cc^4_3+\frac{\alpha}{20}|\sqrt{\phi^2+\eta^2}\nabla^4 u|^2_2.
\end{split}
\end{equation}
For the term $I_B$, from integration by parts, we deduce that
\begin{equation}\label{sszhu1016kls}
\begin{split}
I_B=&-2\int \Big(\partial^{\zeta^2}_x\nabla \phi \partial^{\zeta^3}_x\phi)\Big) \cdot  Q(v) \cdot  \partial^{\zeta+\zeta^1}_xu\\
&=2\int \partial_x^{\zeta^1} \Big[\Big(\partial^{\zeta^2}_x\nabla \phi \partial^{\zeta^3}_x\phi)\Big) \cdot  Q(v)\Big] \cdot  \partial^{\zeta}_xu\\
\leq & C\int \Big(|\nabla^3 u||\nabla v|\big(|\nabla^2 \phi|^2+|\nabla \phi| |\nabla^3 \phi|\big)+|\nabla^3 u||\nabla^2 v||\nabla^2 \phi||\nabla \phi|\Big)\\
\leq& C\Big(|\nabla^2 \phi|_3|\nabla^2 \phi|_6|\nabla^3 u|_2|\nabla v|_\infty+|\nabla^3 u|_2|\nabla v|_\infty|\nabla^3 \phi|_2|\nabla \phi|_\infty\\
&\quad+|\nabla^3 u|_2|\nabla^2 v|_3|\nabla^2 \phi|_6|\nabla \phi|_\infty\Big)\\
\leq&  Cc^3_3|\nabla^3 u|_2.
\end{split}
\end{equation}
Similarly to $I_B$, we have
\begin{equation}\label{sszhu1016kl}
I_C+I_D \leq  Cc^3_3|\nabla^3 u|_2,
\end{equation}
which, combined with (\ref{szhu1015})-(\ref{sszhu1016kl}), implies that
\begin{equation}\label{szhu10177}
\begin{split}
I_{14}\leq \frac{\alpha}{10}|\sqrt{\phi^2+\eta^2} \nabla^4 u|^2_2+ Cc^4_3 |u|_{D^3}+Cc^{4}_3.
\end{split}
\end{equation}

So, from  (\ref{szhu101})-(\ref{szhu1011}) and (\ref{szhu10177}), we arrive at
\begin{equation}\label{zhu1021}
\begin{split}
&\frac{d}{dt}|u|^2_{D^3}+\int (\phi^2+\eta^2)| \nabla^4 u|^2
\leq  Cc^{4}_3 |u|^2_{D^3}+Cc^{2K+4}_3.
\end{split}
\end{equation}
Then from Gronwall's inequality, we have
\begin{equation}\label{zhu10222}
\begin{split}
&|u(t)|^2_{D^3}+\int_0^t |\sqrt{\phi^2+\eta^2} \nabla^4 u|^2_2\text{d}s \leq  C\big(|u_0|^2_{D^3}+c^{2K+4}_3t\big)\exp (Cc^{4}_3 t)\leq Cc^{2}_0,
\end{split}
\end{equation}
for $0\leq t\leq T_3=\min\{T^*,(1+t)^{-2K-4}\}$.\\

\noindent\underline{Step 2}. The estimate for $|\partial^\zeta_x u_t|_2$  when $1\leq |\zeta|\leq 2$. First, from (\ref{zhu1020ss}) we  have
\begin{equation}\label{szhu1020ss}
\begin{split}
|u_t|_{D^1}\leq&  C\left(c^{K}_2+c_0|\sqrt{\phi^2+\eta^2} \nabla^3 u|_2\right)\leq  C\left(c^{K}_2+c^3_0\right)\leq Cc^{K}_2.
\end{split}
\end{equation}
Similarly, from the momentum equations we also have
\begin{equation}\label{szhu1020ss}
\begin{split}
|u_t|_{D^2}=&\Big|v\cdot\nabla u+\frac{A\gamma}{\gamma-1}\nabla \phi^{\frac{2\gamma-2}{\delta-1}}+(\phi^2+\eta^2) Lu-\nabla \phi^2 \cdot Q(v)\Big|_{D^2}\\
\leq & C\Big(\|v\|_3\|\nabla u\|_2+|\phi|^{\frac{2\gamma-2}{\delta-1}-1}_\infty|\nabla^3\phi|_2+|\phi|^{\frac{2\gamma-2}{\delta-1}-2}_\infty|\nabla \phi|_3|\nabla^2 \phi|_6\\
&+|\phi|^{\frac{2\gamma-2}{\delta-1}-3}_\infty|\nabla \phi|^2_\infty|\nabla \phi|_2+|\sqrt{\phi^2+\eta^2}|_\infty|\sqrt{\phi^2+\eta^2} \nabla^4 u|_2\\
&+(|\phi|_\infty+\|\nabla\phi\|_2)^2(\|u\|_3+\|v\|_3)\Big)\\
\leq &  C\left(c^{K}_3+c_0|\sqrt{\phi^2+\eta^2} \nabla^4 u|_2\right),
\end{split}
\end{equation}
which  implies that
$$
\int_{0}^{T_3} |u_t|^2_{D^2} \text{d}s \leq C\int_0^{T_3} \left(c^{2K}_3+c^2_0|\sqrt{\phi^2+\eta^2} \nabla^4 u|^2_2\right) \text{d}s\leq Cc^{4}_0.
$$
\end{proof}

Combining the estimates obtained in Lemmas \ref{f2}-\ref{s5}, we have
\begin{equation}\label{jkkll}
\begin{split}
1+\overline{\phi}^2+|\phi(t)|^2_\infty+\|\phi(t)-\overline{\phi}\|^2_3\leq& Cc^2_0,\\
|\phi_t(t)|^2_2 \leq Cc^4_1,\quad |\phi_t(t)|^2_{D^1} \leq Cc^4_2,\quad |\phi_t(t)|^2_{D^2} \leq& Cc^4_3,\\
\|u(t)\|^2_{ 1}\leq Cc^{2}_0,\quad |u(t)|^2_{ D^2}+|u_t(t)|^2_{2}+\int_{0}^{t}\Big(|\phi \nabla^3 u|^2_{2}+|u_t|^2_{D^1}\Big)\text{d}s\leq& Cc^{2K}_1,\\
|u(t)|^2_{ D^3}+|u_t(t)|^2_{D^1}+\int_{0}^{t}\Big(|\phi \nabla^4 u|^2_{2}+|u_t|^2_{D^2}\Big)\text{d}s\leq& Cc^{2K}_2,
\end{split}
\end{equation}
for
$0 \leq t \leq \min(T^*,(1+c_3)^{-2K-4})$. Therefore, if we define the constants $c_i$ ($i=1,2,3,4$) and $T^*$ by
\begin{equation}\label{dingyi}
\begin{split}
&c_1=C^{\frac{1}{2}}c_0, \quad  c_2=C^{\frac{1}{2}}c^K_1=C^{\frac{K+1}{2}}c^{K}_0,\quad c_3=C^{\frac{1}{2}}c^{K}_2=C^{\frac{K^2+K+1}{2}}c^{K^2}_0,\\
& c_4= C^{\frac{1}{2}}c^2_3=C^{K^2+K+\frac{3}{2}}c^{2K^2}_0, \quad \text{and} \quad T^*=\min(T, (1+c_3)^{-2K-4}),
\end{split}
\end{equation}
then we conclude that
\begin{equation}\label{jkk}
\begin{split}
1+\overline{\phi}^2+\sup_{0\leq t \leq T^*}\big(|\phi(t)|^2_\infty+\|\phi(t)-\overline{\phi}\|^2_3+\|u(t)\|^2_{ 1}\big)\leq& c^{2}_1,\\
\sup_{0\leq t \leq T^*}\big(|\phi_t(t)|^2_{2}+ |u(t)|^2_{ D^2}+|u_t(t)|^2_{2}\big)+\int_{0}^{T^*}\Big(|\phi \nabla^3 u|^2_{2}+|u_t|^2_{D^1}\Big)\text{d}t\leq& c^{2}_2,\\
\text{ess}\sup_{0\leq t \leq T^*}\big(|\phi_t(t)|^2_{D^1}+|u(t)|^2_{ D^3}+|u_t(t)|^2_{D^1}\big)+\int_{0}^{T^*}\Big(|\phi \nabla^4 u|^2_{2}+|u_t|^2_{D^2}\Big)\text{d}t\leq&c^{2}_3,\\
 \sup_{0\leq t \leq T^*}|\phi_t(t)|^2_{D^2} \leq& c^2_4.
\end{split}
\end{equation}

\subsection{Passing to the  limit $\eta\rightarrow 0$} With the help of the $\eta$-independent estimates established in (\ref{jkk}), we now establish the local existence result for the following linearized problem
without artificial viscosity under the assumption $\phi_0\geq 0$,

\begin{equation}\label{li4*}
\begin{cases}
\displaystyle
\phi_t+v\cdot \nabla \phi+\frac{\delta-1}{2}\psi \text{div}v=0,\\[9pt]
\displaystyle
u_t+v\cdot\nabla u +\frac{A\gamma}{\gamma-1} \nabla \phi^{\frac{2\gamma-2}{\delta-1}}+\phi^2 Lu=\nabla \phi^2 \cdot Q(v),\\[9pt]
(\phi,u)|_{t=0}=(\phi_0(x),u_0(x)),\quad x\in \mathbb{R}^3,\\[9pt]
\displaystyle
(\phi,u)\rightarrow (\overline{\phi},0),\quad \text{as}\quad  |x|\rightarrow +\infty,\quad t>0.
 \end{cases}
\end{equation}

\begin{lemma}\label{lem1q}
 Assume that the initial data  satisfy (\ref{th78qq}).
Then there exists a unique strong solution $(\phi, u)(x,t)$ to (\ref{li4*}) such that
\begin{equation}\label{regghq}\begin{split}
& \phi-\overline{\phi} \in C([0,T^*]; H^3), \ \phi _t \in C([0,T^*]; H^2),\\
& u\in C([0,T^*]; H^{s'})\cap L^\infty([0,T^*]; H^3),\quad \phi \nabla^4 u \in L^2([0,T^*] ; L^2), \\
& u_t \in C([0,T^*]; H^1)\cap L^2([0,T^*] ; D^2)
\end{split}
\end{equation}
for any constant $s'\in[2,3)$. Moreover,  $(\phi,u)$ also satisfies the  a priori estimates in (\ref{jkk}).
\end{lemma}
\begin{proof} We shall prove the existence, uniqueness and time continuity in three steps.\\

\noindent\underline{Step 1}. Existence.
From Lemma \ref{lem1}, for every $\eta>0$,  there exists a unique strong solution $(\phi^\eta,  u^\eta)(x,t)$  to the linearized problem (\ref{li4}) satisfying  the estimates in  (\ref{jkk}), which are independent of the artificial viscosity coefficient $\eta$.

By virtue or  the uniform estimates in (\ref{jkk}) independent of  $\eta$ and  the compactness in Lemma \ref{aubin} (see \cite{jm}), we know that for any $R> 0$,  there exists a subsequence of solutions (still denoted by) $(\phi^\eta,  u^\eta)$, which  converges  to  a limit $(\phi,  u)$ in the following  strong sense:
\begin{equation}\label{ert}\begin{split}
&(\phi^\eta,u^\eta) \rightarrow (\phi,u) \quad \text{ in } \ C([0,T^*];H^2(B_R)),\quad \text{as}\ \eta\rightarrow 0.
\end{split}
\end{equation}

Again, using the uniform estimates in (\ref{jkk}) independent of  $\eta$, we also know that  there exists a subsequence (of subsequence chosen above) of solutions (still denoted by) $(\phi^\eta,  u^\eta)$, which    converges to  $(\phi,  u)$ in the following  weak  or  weak* sense:
\begin{equation}\label{ruojixian}
\begin{split}
(\phi^\eta, u^\eta)\rightharpoonup  (\phi,u) \quad &\text{weakly* \ in } \ L^\infty([0,T^*];H^3(\mathbb{R}^3)),\\
\phi^\eta_t\rightharpoonup  \phi_t \quad &\text{weakly* \ in } \ L^\infty([0,T^*];H^2(\mathbb{R}^3)),\\
 u^\eta_t\rightharpoonup  u_t \quad &\text{weakly* \ in } \ L^\infty([0,T^*];H^1(\mathbb{R}^3)),\\
 u^\eta_t\rightharpoonup  u_t \quad &\text{weakly \ \  in } \ \ L^2([0,T^*];D^2(\mathbb{R}^3)).
\end{split}
\end{equation}

Combining  the strong convergence in (\ref{ert}) and the weak convergence in (\ref{ruojixian}), we easily obtain that  $(\phi,  u) $ also satisfies the local estimates in (\ref{jkk}) and
\begin{equation}\label{ruojixian1}
\begin{split}
\phi^\eta \nabla^4 u^\eta \rightharpoonup  \phi\nabla^4 u \quad &\text{weakly \ in } \ L^2(\mathbb{R}^3 \times [0,T^*]).
\end{split}
\end{equation}

Now we  are going to show that $(\phi,  u) $ is a weak solution in the sense of distribution  to the linearized problem (\ref{li4*}).
Multiplying $(\ref{li4})_2$ by  test function  $w(t,x)=(w^1,w^2,w^3)\in C^\infty_c (\mathbb{R}^3 \times [0,T^*))$ on both sides, and integrating over $\mathbb{R}^3 \times [0,T^*]$, we have

\begin{equation}\label{zhenzheng}
\begin{split}
&\int_0^t \int  \Big(u^\eta\cdot w_t - (v\cdot \nabla) u^\eta \cdot w +\frac{A\gamma}{\gamma-1}(\phi^\eta)^{\frac{2\gamma-2}{\delta-1}}\text{div}w \Big)\text{d}x\text{d}s\\
=&-\int u_0 \cdot w(0,x)+\int_0^t \int \Big(((\phi^\eta)^2+\eta^2) Lu^\eta\cdot w-\nabla (\phi^\eta)^2 \cdot Q(v) \cdot w\Big) \text{d}x\text{d}s.
\end{split}
\end{equation}
Combining  the strong convergence in (\ref{ert}) and  the weak convergences in (\ref{ruojixian})-(\ref{ruojixian1}),  and  letting $\eta\rightarrow 0$ in (\ref{zhenzheng}),  we  have
\begin{equation}\label{zhenzhengxx}
\begin{split}
&\int_0^t \int \Big(u \cdot w_t - (v\cdot \nabla) u \cdot w +\frac{A\gamma}{\gamma-1} \phi^{\frac{2\gamma-2}{\delta-1}}\text{div}w \Big)\text{d}x\text{d}s\\
&=-\int u_0 \cdot w(0,x)+\int_0^t \int \Big(\phi^2 Lu \cdot w
-\nabla \phi^2 \cdot Q(v) \cdot w \Big)\text{d}x\text{d}s.
\end{split}
\end{equation}
So it is obvious that $(\phi,  u) $ is a weak solution in the sense of distribution  to the linearized problem (\ref{li4*}), satisfying the  following regularities
\begin{equation}\label{zheng}
\begin{split}
& \phi-\overline{\phi} \in L^\infty([0,T^*]; H^3), \ \phi _t \in L^\infty([0,T^*]; H^2), \\
& u\in L^\infty([0,T^*]; H^3),\quad \phi \nabla^4 u \in L^2([0,T^*] ; L^2), \\
& u_t \in L^\infty([0,T^*]; H^1)\cap L^2([0,T^*] ; D^2),
\end{split}
\end{equation}
where we used the lower semi-continuity of various norms in the weak or weak* convergence in (\ref{ruojixian})-(\ref{ruojixian1}). Therefore, this weak solutions $(\phi,  u) $ of \eqref{li4*} is actually
a strong one.\\

\noindent\underline{Step 2}. Uniqueness.  Let $(\phi_1,u_1)$ and $(\phi_2,u_2)$ be two solutions obtained in the above step. For $\varphi=\phi_1-\phi_2$, we have from $(\ref{li4*})_1$ that

\begin{equation}\label{qaq}
\begin{split}
{\varphi}_t+v\cdot \nabla {\varphi}=0,
\end{split}
\end{equation}
which immediately implies  that ${\varphi}=0$  in $\mathbb{R}^3$ with zero initial data.

For $\overline{u}=u_1-u_2$, from $(\ref{li4*})_2$, using the fact $\phi_1=\phi_2$, it is clear that
\begin{equation}\label{aq}
\begin{split}
\overline{u}_t+v\cdot \nabla \overline{u}-\phi^2_1 L\overline{u}=0.
\end{split}
\end{equation}
Multiplying  (\ref{aq}) by $\overline{u}$  on both sides, and integrating  over $\mathbb{R}^3$,  we have
 \begin{equation}\label{bq}
\begin{split}
\frac{d}{dt} |\overline{u}|^2_2 +|\phi_1 \nabla \overline{u}|^2_2
\leq& C|\nabla v|_\infty |\overline{u}|^2_2+C|\overline{u}|_2 |\nabla \phi_1|_\infty |\phi_1 \nabla \overline{u}|_2\\
\leq & \frac{1}{10}|\phi_1 \nabla \overline{u}|^2_2+C(|\nabla v|_\infty +|\nabla \phi_1|^2_\infty)|\overline{u}|^2_2.
\end{split}
\end{equation}
Now, the Gronwall's inequality, along with zero initial data of $\overline{u}$ implies that $\overline{u}=0$ in $\mathbb{R}^3$. This completes the proof of uniqueness.\\

\noindent\underline{Step 3}. Time continuity.   For $\phi$,  the regularities in (\ref{zheng}) and the classical Sobolev imbedding theorem infer that
\begin{equation}\label{liu02}
\phi-\overline{\phi} \in C([0,T^*];H^2) \cap C([0,T^*]; \text{weak}-H^3).
\end{equation}

Using the same arguments as  in the proof of Lemma \ref{f2}, we have
\begin{equation}\label{qgb}\begin{split}
&\|\phi(t)-\overline{\phi}\|^2_3\\
\leq& \Big(\|\phi_0-\overline{\phi}\|^2_{3}+C\int_0^t \big(\|\nabla \psi\|^2_2\|v\|^2_3+|\phi\nabla^4 v|^2_2\big)\text{d}s\Big) \exp\Big(C\int_0^t \big(\| v(s)\|_{3}+1\big)\text{d}s\Big),
\end{split}
\end{equation}
which implies that
\begin{equation}\label{liu03}
\displaystyle \lim_{t\rightarrow 0}\sup \|\phi(t)-\overline{\phi}\|_3 \leq \|\phi_0-\overline{\phi}\|_3.
\end{equation}
From  Lemma \ref{zheng5} and  (\ref{liu02}), we know that  $\phi$ is right continuous at $t=0$ in $H^3$ space. The time reversibility of the equation $(\ref{li4*})_1$ yields
\begin{equation}\label{xian}
\phi-\overline{\phi} \in C([0,T^*];H^3).
\end{equation}

For $\phi_t$, we note that
\begin{equation}\label{liu04}
\phi_t=-v\cdot \nabla \phi-\frac{\delta-1}{2}\psi \text{div}v.
\end{equation}
On the other hand, since
\begin{equation}\label{liu05}
\psi \nabla v\in L^2([0,T^*];H^3),\quad (\psi \nabla v)_t  \in L^2([0,T^*];H^1),
\end{equation}
using Sobolev imbedding theorem, we have
\begin{equation}\label{liu06}
\psi \nabla v\in C([0,T^*];H^2),
\end{equation}
which implies that
$$
\phi_t \in C([0,T^*];H^2).
$$

For velocity $u$, from the regularities shown in (\ref{zheng}) and Sobolev imbedding theorem, we know that
\begin{equation}\label{zheng1}
\begin{split}
 u\in C([0,T^*]; H^2)\cap  C([0,T^*]; \text{weak}-H^3).
\end{split}
\end{equation}
Then from  Lemma \ref{gag111},  for any $s'\in [2,3)$, we have

$$
\|u\|_{s'} \leq C_3 \|u\|^{1-\frac{s'}{3}}_0 \|u\|^{\frac{s'}{3}}_3.
$$
This, together with the upper bound estimates shown in (\ref{jkk}) and the time continuity (\ref{zheng1}), yields
\begin{equation}\label{zheng2}
\begin{split}
 u\in C([0,T^*]; H^{s'}).
\end{split}
\end{equation}

Finally, we consider $u_t$.  Noting that
\begin{equation}\label{zheng3}
u_t=-v\cdot\nabla u -\frac{2A\gamma}{\delta-1}\phi^{\frac{2r-\delta-1}{\delta-1}}\nabla \phi+\phi^2 Lu+\nabla \phi^2\cdot Q(v),
\end{equation}
where
$$
Q(v)=\frac{\delta}{\delta-1}\left(\alpha(\nabla v+(\nabla v)^\top)+\beta \text{div}v \mathbb{I}_3\right)\in L^2([0,T^*];H^2),
$$
we then have from (\ref{zheng}) that
\begin{equation}\label{zheng4}
\phi^2 Lu\in L^2([0,T^*];H^2), \quad (\phi^2 Lu)_t\in L^2([0,T^*];L^2),
\end{equation}
which means that
\begin{equation}\label{gong4}
\phi^2 Lu\in C([0,T^*];H^1).
\end{equation}
Combining (\ref{vg}), (\ref{xian}), (\ref{zheng2}) and  (\ref{gong4}), we deduce that
$$
u_t \in C([0,T^*];H^1).
$$
We thus complete the proof of this lemma.

\end{proof}

\subsection{Proof of Theorem \ref{ths1}}  Now we turn to the nonlinear problem (\ref{eq:cccq})-(\ref{sfanb1}). Our proof  is based on the classical iteration scheme and the existence results for the linearized problem obtained  in  subsection $3.4$. We first define constants $c_{0}$ and  $c_{1}$, $c_2$, $c_3$, $c_4$ as in subsection 3.3. Assume that

\begin{equation*}\begin{split}
&2+\overline{\phi} +|\phi_0|_{\infty}+\|(\phi_0-\overline{\phi}, u_0)\|_{3}\leq c_0.
\end{split}
\end{equation*}
Let $(\phi^0,  u^0)$
be the solution to the  system
\begin{equation}\label{zheng6}
\begin{cases}
 Y_t+u_0 \cdot \nabla Y=0 \quad \text{in} \quad (0,+\infty)\times \mathbb{R}^3,\\[10pt]
Z_t- Y^2\triangle Z=0 \quad \ \text{in} \quad  (0,+\infty)\times \mathbb{R}^3 ,\\[10pt]
(Y,Z)|_{t=0}=(\phi_0,u_0) \quad \text{in} \quad \mathbb{R}^3,\\[10pt]
(Y,Z)\rightarrow (\overline{\phi},0) \quad \text{as } \quad |x|\rightarrow +\infty,\quad t> 0,
\end{cases}
\end{equation}
 with the regularities
\begin{equation}\label{zheng6}
\begin{split}
&\phi^0-\overline{\phi}  \in C([0,T^*];H^{3}),\quad \phi^0 \nabla^4 u^0 \in L^2([0,T^*];L^2), \\
& u^0\in C([0,T^*];H^{s'})\cap  L^\infty([0,T^*];H^3) \quad \text{for\ any } \ s' \in [2,3).
\end{split}
\end{equation}
Due to the regularity of $(\phi^0, u^0)(x)$, there is a positive time $T^{**}\in (0,T^*]$ such that
\begin{equation}\label{jizhu}
\begin{split}
1+\overline{\phi}^2+\sup_{0\leq t \leq T^{**}}\big(|\phi^0(t)|^2_\infty+\|\phi^0(t)-\overline{\phi}\|^2_3+\|u^0(t)\|^2_{ 1}\big)\leq& c^{2}_1,\\
\sup_{0\leq t \leq T^{**}}\big(|\phi^0_t(t)|^2_{2}+ |u^0(t)|^2_{ D^2}+|u^0_t(t)|^2_{2}\big)+\int_{0}^{T^{**}}\Big(|\phi^0 \nabla^3 u^0|^2_{2}+|u^0_t|^2_{D^1}\Big)\text{d}t\leq& c^{2}_2,\\
\text{ess}\sup_{0\leq t \leq T^{**}}\big(|\phi^0_t|^2_{D^1}+|u^0|^2_{ D^3}+|u^0_t|^2_{D^1}\big)(t)+\int_{0}^{T^{**}}\Big(|\phi^0 \nabla^4 u^0|^2_{2}+|u^0_t|^2_{D^2}\Big)\text{d}t\leq &c^{2}_3,\\
 \sup_{0\leq t \leq T^{**}}|\phi^0_t|^2_{D^2} \leq& c^2_4.
\end{split}
\end{equation}
We now give the proof of Theorem \ref{ths1}.
\begin{proof} We prove the existence, uniqueness and time continuity in three steps.\\

\noindent\underline{Step 1}: Existence. Letting $(\psi,v)=(\phi^0,u^0)$, we first define $(\phi^1,u^1)$ as a strong solution to problem (\ref{li4*}). Then we construct approximate solutions $(\phi^{k+1},u^{k+1})$ inductively, as follows: assuming that $(\phi^{k}, u^{k})$ was defined for $k\geq 1$, let $(\phi^{k+1},  u^{k+1})$  be the unique solution to problem (\ref{li4*}) with $(\psi,v)$ replaced by $ (\phi^k,u^{k})$, i.e., $(\phi^{k+1},  u^{k+1})$  is the unique solution of the following problem
\begin{equation}\label{li6}
\begin{cases}
\displaystyle
\ \phi^{k+1}_t+u^{k}\cdot \nabla \phi^{k+1}+\frac{\delta-1}{2}\phi^{k}\text{div} u^{k}=0,\\[9pt]
\displaystyle
\ u^{k+1}_t+u^{k}\cdot\nabla u^{k+1} +\frac{2A\gamma}{\delta-1}\Phi^{k+1}\nabla \phi^{k+1}+(\phi^{k+1})^2Lu^{k+1}=\nabla (\phi^{k+1})^2 \cdot Q(u^{k}),\\[9pt]
\ (\phi^{k+1}, u^{k+1})|_{t=0}=(\phi_0,u_0),\quad x\in \mathbb{R}^3,\\[9pt]
\displaystyle
(\phi^{k+1},u^{k+1})\rightarrow (\overline{\phi},0),\quad \text{as}\quad  |x|\rightarrow +\infty,\quad t>0,
 \end{cases}
\end{equation}
where $
\Phi^{k+1}=(\phi^{k+1})^{\frac{2\gamma-\delta-1}{\delta-1}}.$

From the estimates shown in Section $3.4$, we  know  that the sequence $(\phi^{k},  u^{k})$ satisfies the  uniform a priori estimates in (\ref{jkk}) for $0 \leq t \leq T^{**}$.

Now we  prove the  convergence of the whole sequence  $(\phi^k,  u^k)$ of approximate solutions to a limit $(\phi, u)$  in some strong sense.
Let
\begin{equation*}\begin{split}
&\overline{\phi}^{k+1}=\phi^{k+1}-\phi^k,\quad  \overline{u}^{k+1}=u^{k+1}-u^k,\\
&\overline{\Phi}^{k+1}=\Phi^{k+1}-\Phi^k=(\phi^{k+1})^{\frac{2\gamma-\delta-1}{\delta-1}}-(\phi^k)^{\frac{2\gamma-\delta-1}{\delta-1}}.
\end{split}
\end{equation*}
 Then from (\ref{li6}) we have
 \begin{equation}
\label{eq:1.2w}
\begin{cases}
\displaystyle
\ \  \overline{\phi}^{k+1}_t+u^k\cdot \nabla\overline{\phi}^{k+1} +\overline{u}^k\cdot\nabla\phi ^{k}+\frac{\delta-1}{2}(\overline{\phi}^{k} \text{div}u^{k-1} +\phi ^{k}\text{div}\overline{u}^k)=0,\\[12pt]
\displaystyle
\ \ \overline{u}^{k+1}_t+ u^k\cdot\nabla \overline{u}^{k+1}+ \overline{u}^{k} \cdot \nabla u^{k}+\frac{2A\gamma}{\delta-1}(\Phi^{k+1}\nabla \overline{\phi}^{k+1}+\overline{\Phi}^{k+1}\nabla \phi^k) \\[12pt]
\displaystyle
=-(\phi^{k+1})^2L\overline{u}^{k+1}- \overline{\phi}^{k+1}(\phi^{k+1}+\phi^k)Lu^k\\[12pt]
\displaystyle
\ \ +\nabla (\overline{\phi}^{k+1}(\phi^{k+1}+\phi^k))\cdot Q(u^k)+\nabla (\phi^k)^2\cdot (Q(u^{k})-Q(u^{k-1})).
\end{cases}
\end{equation}

First, multiplying $(\ref{eq:1.2w})_1$ by $2\overline{\phi}^{k+1}$ and integrating  over $\mathbb{R}^3$, we have
\begin{equation*}
\begin{split}
\frac{d}{dt}|\overline{\phi}^{k+1}|^2_2=& -2\int\Big(u^k\cdot \nabla\overline{\phi}^{k+1} +\overline{u}^k\cdot\nabla\phi ^{k}+\frac{\delta-1}{2}(\overline{\phi}^{k} \text{div}u^{k-1} +\phi ^{k}\text{div}\overline{u}^k)\Big)\overline{\phi}^{k+1}\\
\leq& C\Big(|\nabla u^k|_\infty|\overline{\phi}^{k+1}|^2_2+ |\overline{\phi}^{k+1}|_2|\overline{u}^k|_2|\nabla \phi^k|_\infty\\
&\qquad+|\overline{\phi}^{k}|_2|\nabla u^{k-1}|_\infty|\overline{\phi}^{k+1}|_2+|\phi^k \text{div}\overline{u}^k|_2 |\overline{\phi}^{k+1}|_2\Big),
\end{split}
\end{equation*}
which yields that ( for $0<\nu \leq \frac{1}{10}$ is a constant)
\begin{equation}\label{go64}\begin{cases}
\displaystyle
\frac{d}{dt}|\overline{\phi}^{k+1}(t)|^2_2\leq A^k_\nu(t)|\overline{\phi}^{k+1}(t)|^2_2+\nu\Big( |\overline{u}^k(t)|^2_2+|\overline{\phi}^k(t)|^2_2+|\phi^k \text{div}\overline{u}^k(t)|^2_2\Big),\\[12pt]
\displaystyle
A^k_\nu(t)=C\Big(|\nabla u^k|_{\infty}+\frac{1}{\nu}|\nabla u^{k-1}|^2_{\infty}+\frac{1}{\nu}|\nabla \phi^{k}|^2_{\infty}+\frac{1}{\nu} \Big).
\end{cases}
\end{equation}
From (\ref{jkk}), we know
$$
\ \int_0^t A^k_\nu(s)\text{d}s\leq C_{\nu}t,\quad \text{for}\quad  t\in[0,T^{**}],
$$
where $C_{\nu}$ is a positive constant  depending on $\nu$ and  constant $C$.

Second, we multiply $(\ref{eq:1.2w})_2$ by $2\overline{u}^{k+1}$ and integrate  over $\mathbb{R}^3$ to
find
\begin{equation}\label{zheng8}
\begin{split}
&\frac{d}{dt}|\overline{u}^{k+1}|^2_2+2\alpha|\phi^{k+1}\nabla\overline{u}^{k+1} |^2_2+2(\alpha+\beta)|\phi^{k+1}\text{div}\overline{u}^{k+1} |^2_2\\
=& -2\int  \Big( u^k \cdot \nabla \overline{u}^{k+1}+\overline{u}^{k} \cdot \nabla u^{k}\Big) \cdot \overline{u}^{k+1}-2\int \frac{2A\gamma}{\delta-1}(\Phi^{k+1}\nabla \overline{\phi}^{k+1}+\overline{\Phi}^{k+1}\nabla \phi^k)  \cdot \overline{u}^{k+1}\\
&-4\alpha\int \phi^{k+1} \nabla \phi^{k+1}\cdot \nabla \overline{u}^{k+1}\cdot \overline{u}^{k+1}-4(\alpha+\beta)\int  \phi^{k+1} \nabla \phi^{k+1}\cdot  \overline{u}^{k+1} \text{div} \overline{u}^{k+1}\\
&-2\int  \overline{\phi}^{k+1}(\phi^{k+1}+\phi^k)Lu^k \cdot  \overline{u}^{k+1}
+2\int \nabla (\overline{\phi}^{k+1}(\phi^{k+1}+\phi^k))\cdot Q(u^k) \cdot  \overline{u}^{k+1} \\
&+2\int \nabla (\phi^k)^2\cdot (Q(u^{k})-Q(u^{k-1})) \cdot  \overline{u}^{k+1}=:\sum_{i=1}^{9} J_i.
\end{split}
\end{equation}

 We now  estimate $J_i$ $(i=1, \cdots, 9)$ term by term. For the term $J_1$, we see from integration by parts that
\begin{equation}\label{ya1}
\begin{split}
J_1= -2\int u^k \cdot \nabla \overline{u}^{k+1} \cdot \overline{u}^{k+1}\leq C|\nabla u^k|_\infty | \overline{u}^{k+1}|^2_2.
\end{split}
\end{equation}

For $J_2$, it is easy to show that
\begin{equation}\label{ya1}
\begin{split}
J_2=& -2\int \overline{u}^{k} \cdot \nabla u^{k}\cdot \overline{u}^{k+1}\\
\leq & C|\nabla u^k|_\infty | \overline{u}^{k}|_2 | \overline{u}^{k+1}|_2
\leq  \frac{C}{\nu} |\nabla u^k|^2_\infty | \overline{u}^{k+1}|^2+\nu | \overline{u}^{k}|^2_2.
\end{split}
\end{equation}

Applying integration by parts for $J_3$, we have
\begin{equation}\label{ya2}
\begin{split}
J_3=&-2\int \frac{2A\gamma}{\delta-1}\Phi^{k+1}\nabla \overline{\phi}^{k+1}\cdot \overline{u}^{k+1}\\
=& \frac{4A\gamma}{\delta-1} \int \big(\Phi^{k+1}  \overline{\phi}^{k+1}\text{div} \overline{u}^{k+1}+ \overline{\phi}^{k+1}\nabla\Phi^{k+1}\cdot \overline{u}^{k+1}\big)\\
\leq &C\Big( |\phi^{k+1}|^{\frac{2\gamma-2\delta}{\delta-1}}_\infty | \overline{\phi}^{k+1}|_2 |\phi^{k+1}\text{div}\overline{u}^{k+1} |_2+ |\overline{\phi}^{k+1}|_2|\nabla\Phi^{k+1}|_\infty | \overline{u}^{k+1}|_2\Big)\\
\leq &  C\Big(|\phi^{k+1}|^{\frac{4\gamma-4\delta}{\delta-1}}_\infty| \overline{\phi}^{k+1}|^2_2+ |\overline{\phi}^{k+1}|^2_2+|\nabla\Phi^{k+1}|^2_\infty | \overline{u}^{k+1}|^2_2\Big) +\frac{\alpha}{20}  |\phi^{k+1}\nabla \overline{u}^{k+1} |^2_2.
\end{split}
\end{equation}

$J_4$ is estimated directly below,
\begin{equation}\label{ya3}
\begin{split}
J_4=&-2\int \frac{2A\gamma}{\delta-1}\overline{\Phi}^{k+1}\nabla \phi^k \cdot \overline{u}^{k+1}\\
\leq & C|\overline{\Phi}^{k+1}|_2 |\nabla \phi^k|_\infty | \overline{u}^{k+1}|_2\\
\leq & C \Big(\big||\phi^{k+1}|+|\phi^{k}|\big|^{\frac{4\gamma-4\delta}{\delta-1}}_\infty|\overline{\phi}^{k+1}|^2_2+ |\nabla \phi^k|^2_\infty | \overline{u}^{k+1}|^2_2\Big),\\
\end{split}
\end{equation}
where we used the fact that
$$
|\overline{\Phi}^{k+1}|_2=|(\phi^{k+1})^{\frac{2\gamma-\delta-1}{\delta-1}}-(\phi^k)^{\frac{2\gamma-\delta-1}{\delta-1}}|_2\leq C \big||\phi^{k+1}|+|\phi^{k}|\big|^{\frac{2\gamma-2\delta}{\delta-1}}_\infty|\overline{\phi}^{k+1}|_2.
$$
Similarly, we are able to treat terms $J_5$-$J_7$ in the following way,
\begin{equation}\label{ya4}
\begin{split}
J_5=&-4\alpha\int \phi^{k+1} \nabla \phi^{k+1}\cdot \nabla \overline{u}^{k+1}\cdot \overline{u}^{k+1}\\
\leq & C|\nabla \phi^{k+1}|_\infty |\phi^{k+1}\nabla \overline{u}^{k+1}|_2|\overline{u}^{k+1}|_2\\
\leq&  C |\nabla \phi^{k+1}|^2_\infty|\overline{u}^{k+1}|^2_2+\frac{\alpha}{20} |\phi^{k+1}\nabla \overline{u}^{k+1}|^2_2,\\
J_6=&-4(\alpha+\beta)\int  \phi^{k+1} \nabla \phi^{k+1}\cdot  \overline{u}^{k+1} \text{div} \overline{u}^{k+1}\\
\leq& C|\nabla \phi^{k+1}|_\infty |\phi^{k+1}\nabla \overline{u}^{k+1}|_2|\overline{u}^{k+1}|_2\\
\leq&  C |\nabla \phi^{k+1}|^2_\infty|\overline{u}^{k+1}|^2_2+\frac{\alpha}{20} |\phi^{k+1}\nabla \overline{u}^{k+1}|^2_2,\\
%
J_7=&2\int  \overline{\phi}^{k+1}(\phi^{k+1}+\phi^k)Lu^k \cdot  \overline{u}^{k+1} \\
\leq & C\Big(| \overline{\phi}^{k+1}|_2| \phi^{k+1}\overline{u}^{k+1} |_6|Lu^k|_3+|\overline{\phi}^{k+1}|_2|\phi^kLu^k|_\infty| \overline{u}^{k+1}|_2\Big)\\
\leq & C\Big(| \overline{\phi}^{k+1}|_2(| \phi^{k+1}\nabla \overline{u}^{k+1} |_2+|\nabla\phi^{k+1} |_\infty|\overline{u}^{k+1} |_2)|Lu^k|_3\\
&\qquad+|\overline{\phi}^{k+1}|_2(\|\nabla \phi^k\|_2\|u^k\|_3+|\phi^k \nabla^4 u^k|_2)| \overline{u}^{k+1}|_2\Big)\\
\leq&C\Big( |Lu^k|^2_3 |\overline{\phi}^{k+1}|^2_2 +| \overline{\phi}^{k+1}|^2_2+|\nabla\phi^{k+1} |^2_\infty|Lu^k|^2_3|\overline{u}^{k+1} |^2_2\\
&\qquad+(\|\nabla \phi^k\|_2\|\nabla^2 u^k\|_1+|\phi^k \nabla^4 u^k|_2)^2|\overline{u}^{k+1} |^2_2\Big)+\frac{\alpha}{20}  | \phi^{k+1}\nabla \overline{u}^{k+1} |^2_2,
\end{split}
\end{equation}
where we used the fact (see Lemma \ref{lem2as}) that
\begin{equation}\label{qianru}
\begin{split}
|\phi^{k}\nabla^2 u^k|_\infty\leq&  C|\phi^{k}\nabla^2 u^k|^{\frac{1}{2}}_6|\nabla (\phi^{k}\nabla^2 u^k)|^{\frac{1}{2}}_6\\
\leq& C|\phi^{k}\nabla^2 u^k|^{\frac{1}{2}}_{D^1}|\nabla (\phi^{k}\nabla^2 u^k)|^{\frac{1}{2}}_{D^1}\\
\leq &C\|\nabla (\phi^{k}\nabla^2 u^k)\|_1\\
\leq &C\Big(|\nabla \phi^k|_\infty\|\nabla^2 u^k\|_1+|\phi^k|_\infty|\nabla^3 u^k|_2+|\phi^k\nabla^4 u^k|_2+|\nabla^2 \phi^k|_6|\nabla^2u^k|_3\Big)\\
\leq &C\Big(\|\nabla\phi^k\|_2\|\nabla^2u^k\|_1+|\phi^k\nabla^4 u^k|_2\Big).
\end{split}
\end{equation}

Now we turn to the tricky term $J_8$. First we have
\begin{equation}\label{ya7}
\begin{split}
J_8=&2\int \nabla (\overline{\phi}^{k+1}(\phi^{k+1}+\phi^k))\cdot Q(u^k) \cdot  \overline{u}^{k+1}\\
=&-2 \int \sum_{i,j}\overline{\phi}^{k+1} (\phi^{k+1}+\phi^k)\partial_i a^{ij}_k\overline{u}^{k+1,j}-2\int  \sum_{i,j} \overline{\phi}^{k+1}(\phi^{k+1}+\phi^k) a^{ij}_k\partial_i \overline{u}^{k+1,j}\\=:& J_{81}+J_{82}+J_{83}+J_{84},
\end{split}
\end{equation}
where $u^{k,j}$ represents the $j$-$th$ component of $u^k$ ($k\geq 1$),
$$\overline{u}^{k,j}=u^{k,j}-u^{k-1,j},\quad \text{for}\quad  k\geq 1, \quad j=1,2,3,$$ and the quantity $a^{ij}_k$ is given by
$$
a^{ij}_k=\frac{\delta}{\delta-1}\Big(\alpha (\partial_i u^{k,j}+\partial_j u^{k,i}\big)+\text{div}u^k \delta_{ij}\Big)\quad \text{for}\ i,\ j=1,2,3,
$$
 where $\delta_{ij}$ is the Kronecker symbol satisfying $\delta_{ij}=1,\ i=j$, and $\delta_{ij}=0$, otherwise.

We are now ready to estimate terms $J_{81}-J_{84} $ one by one. First we have
\begin{equation}\label{ya91}
\begin{split}
J_{81}=&-2 \int \sum_{i,j}\overline{\phi}^{k+1} \phi^{k+1}\partial_i a^{ij}_k\overline{u}^{k+1,j}
\leq  C |\nabla^2 u^k|_3 |\overline{\phi}^{k+1}|_2|\phi^{kQ+1}\overline{u}^{k+1}|_6\\
\leq & C |\nabla^2 u^k|_3 |\overline{\phi}^{k+1}|_2\left(| \phi^{k+1}\nabla \overline{u}^{k+1} |_2+|\nabla\phi^{k+1} |_\infty|\overline{u}^{k+1} |_2\right)\\
\leq& C\left(|\nabla \phi^{k+1}|^2_\infty |\nabla^2 u^k|^2_3 |\overline{\phi}^{k+1}|^2_2+|\overline{u}^{k+1}|^2_2+ |\nabla^2 u^k|^2_3 |\overline{\phi}^{k+1}|^2_2\right)+\frac{\alpha}{20}  | \phi^{k+1}\nabla \overline{u}^{k+1} |^2_2.
\end{split}
\end{equation}
Similarly,  for the terms $I_{82}$-$I_{83}$, using (\ref{qianru}), we get
\begin{equation}\label{ya92}
\begin{split}
J_{82}=&-2 \int \sum_{i,j}\overline{\phi}^{k+1} \phi^k \partial_i a^{ij}_k\overline{u}^{k+1,j}
\leq C|\phi^{k}\nabla^2 u^k|_\infty |\overline{\phi}^{k+1}|_2|\overline{u}^{k+1}|_2\\
\leq& C\left( \left(\|\nabla\phi^k\|_2\|\nabla^2u^k\|_1+|\phi^k\nabla^4 u^k|_2\right)^2   |\overline{\phi}^{k+1}|^2_2+|\overline{u}^{k+1}|^2_2\right),\\
J_{83}=&-2\int  \sum_{i,j} \overline{\phi}^{k+1}\phi^{k+1} a^{ij}_k\partial_i \overline{u}^{k+1,j}
\leq C| \overline{\phi}^{k+1}|_2 |\phi^{k+1}\nabla\overline{u}^{k+1}|_2|\nabla u^k|_\infty \\
\leq & C|\nabla u^k|^2_\infty| \overline{\phi}^{k+1}|^2_2+\frac{\alpha}{20}  |\phi^{k+1}\nabla\overline{u}^{k+1}|^2_2.
\end{split}
\end{equation}
For $J_{84}$, we have
\begin{equation}\label{ya94}
\begin{split}
J_{84}=&-2\int  \sum_{i,j} \overline{\phi}^{k+1}\phi^k a^{ij}_k\partial_i \overline{u}^{k+1,j}\\
=&-2\int  \sum_{i,j} \overline{\phi}^{k+1}(\phi^k-\phi^{k+1}+\phi^{k+1}) a^{ij}_k\partial_i \overline{u}^{k+1,j}\\
\leq &C\left(|\nabla u^k|_\infty |\nabla\overline{u}^{k+1}|_\infty|\overline{\phi}^{k+1}|^2_2+|\nabla u^k|_\infty|\phi^{k+1} \nabla\overline{u}^{k+1}|_2|\overline{\phi}^{k+1}|_2\right)\\
\leq &C\left(|\nabla u^k|_\infty |\nabla\overline{u}^{k+1}|_\infty|\overline{\phi}^{k+1}|^2_2+|\nabla u^k|^2_\infty|\overline{\phi}^{k+1}|^2_2\right)+\frac{\alpha}{20}  |\phi^{k+1} \nabla\overline{u}^{k+1}|^2_2.
\end{split}
\end{equation}
Combining (\ref{ya7})-(\ref{ya94}), we arrive at
\begin{equation}\label{ya95}
\begin{split}
J_8\leq &C\Big(|\overline{u}^{k+1}|^2_2+ (|\nabla \phi^{k+1}|^2_\infty+1)  |\nabla^2 u^k|^2_3 |\overline{\phi}^{k+1}|^2_2+|\nabla u^k|^2_\infty|\overline{\phi}^{k+1}|^2_2\\
&\quad+\big(\|\nabla\phi^k\|_2\|\nabla^2u^k\|_1+|\phi^k\nabla^4 u^k|_2\big)^2  |\overline{\phi}^{k+1}|^2_2+|\nabla u^k|_\infty |\nabla\overline{u}^{k+1}|_\infty|\overline{\phi}^{k+1}|^2_2\Big)\\
&+ \frac{\alpha}{5} |\phi^{k+1} \nabla\overline{u}^{k+1}|^2_2.
\end{split}
\end{equation}

As for the last term $J_9$, it is easy to show that
\begin{equation}\label{ya8}
\begin{split}
J_9=&2\int \nabla (\phi^k)^2\cdot (Q(u^{k})-Q(u^{k-1})) \cdot  \overline{u}^{k+1} \\
\leq&C |\nabla \phi^k|_\infty|\phi^k \nabla \overline{u}^k|_2| \overline{u}^{k+1}|_2
\leq  \frac{C}{\nu}  |\nabla \phi^k|^2_\infty | \overline{u}^{k+1}|^2_2+\nu |\phi^k \nabla \overline{u}^k|^2_2,
\end{split}
\end{equation}

In summary, using \eqref{ya1}-(\ref{ya4}) and (\ref{ya95})-(\ref{ya8}),  (\ref{zheng8}) implies  that
\begin{equation}\label{gogo1}\begin{split}
&\frac{d}{dt}|\overline{u}^{k+1}|^2_2+\alpha |\phi^{k+1}\nabla\overline{u}^{k+1} |^2_2\\
\leq& B^k_\nu(t)|\overline{u}^{k+1}|^2_2+B^k(t)|\overline{\phi}^{k+1}|^2_{2}
+\nu (|\phi^k \nabla \overline{u}^k|^2_2+ | \overline{u}^{k}|^2_2),
\end{split}
\end{equation}
for $0<\nu \leq \frac{1}{10}$ is a constant. Here $B^k_\nu(t)$ and $B^k(t)$ are given by
\begin{equation*}
\begin{split}
\displaystyle
B^k_\nu(t)=&C\Big(1+ |\nabla u^k|_\infty+\frac{1}{\nu} |\nabla u^k|^2_\infty+|\nabla\Phi^{k+1}|^2_\infty+|\nabla\phi^{k}|^2_\infty+|\nabla \phi^{k+1}|^2_\infty \\
\displaystyle
&+ (\|\nabla \phi^k\|_2\|u^k\|_3+|\phi^k \nabla^4 u^k|_2)^2+|\nabla\phi^{k+1} |^2_\infty|Lu^k|^2_3+\frac{1}{\nu}|\nabla \phi^k|^2_\infty\Big)\\
\displaystyle
 B^k(t)=&C\Big(1+|\nabla^2 u^k|^2_3+\big||\phi^{k+1}|+|\phi^k|\big|^{\frac{4\gamma-4\delta}{\delta-1}}_\infty+(\|\nabla \phi^k\|_2\|u^k\|_3+|\phi^k\nabla^4 u^k|_2)^2\\
&  +|\nabla \phi^{k+1}|^2_\infty |\nabla^2 u^k|^2_3+|\nabla u^k|^2_\infty+|\nabla u^k|_\infty |\nabla\overline{u}^{k+1}|_\infty\Big),
\end{split}
\end{equation*}
 satisfying   the following estimate
$$\int_0^t \big(B^k_\nu(s)+B^k(s)\big)\text{d}s\leq C+C_\nu t,\quad  t\in[0,T^{**}]. $$

Denote
\begin{equation*}\begin{split}
\Gamma^{k+1}(t)=&\sup_{s\in [0,t]}|\overline{\phi}^{k+1}(s)|^2_{2}+\sup_{s\in [0,t]}|\overline{u}^{k+1}(s)|^2_2.
\end{split}
\end{equation*}
From (\ref{go64}) and (\ref{gogo1}) we finally have
\begin{equation*}\begin{split}
&\frac{d}{dt}(|\overline{\phi}^{k+1}(t)|^2_{2}+|\overline{u}^{k+1}(t)|^2_2)+\alpha |\phi^{k+1}\nabla\overline{u}^{k+1} |^2_2\\
\leq& E^k_\nu \left(|\overline{\phi}^{k+1}(t)|^2_{2}+|\overline{u}^{k+1}(t)|^2_2\right)+ \nu \left(|\phi^k \nabla \overline{u}^k(t)|^2_2+| \overline{\phi}^{k}(t)|^2_2+ | \overline{u}^{k}(t)|^2_2\right),
\end{split}
\end{equation*}
for some $E^k_\nu$ satisfying  $\displaystyle\int_{0}^{t}E^k_\nu(s)\text{d}s\leq C+C_\nu t$.
Applying Gronwall's inequality, we have
\begin{equation*}\begin{split}
&\Gamma^{k+1}+\int_{0}^{t}\alpha|\phi^{k+1}\nabla\overline{u}^{k+1} |^2_2\text{d}s\\
\leq&   C\nu\int_{0}^{t}   \Big(|\phi^k \nabla \overline{u}^k|^2_2+| \overline{\phi}^{k}|^2_2+ | \overline{u}^{k}|^2_2\Big)\text{d}s\exp{(C+C_\nu t)}\\
\leq & C\nu \left(\int_{0}^{t}   |\phi^k \nabla \overline{u}^k|^2_2\text{d}s+t \sup_{s\in [0,t]} | \overline{\phi}^{k}(s)|^2_2+t \sup_{s\in [0,t]} | \overline{u}^{k}(s)|^2_2\right)\exp{(C_\nu t)}.
\end{split}
\end{equation*}
We choose $\nu>0$ and $T_*\in (0,\min (1,T^{**}))$ small enough such that
$$
4C\nu \leq  \min (\alpha, 1), \quad \text{and}\quad \text{exp}(C_\nu T_*)\leq 2.
$$
Therefore,  we achieve 
\begin{equation}\label{kexi}
\begin{split}
\sum_{k=1}^{\infty}\Big( \Gamma^{k+1}(T_*)+\int_{0}^{T_*} \alpha|\phi^{k+1}\nabla\overline{u}^{k+1} |^2_2\text{d}t\Big)\leq C<+\infty.
\end{split}
\end{equation}
This implies that our approximate solution sequence $(\phi^k, u^k)$ is a Cauchy sequence under the topology
of $L^\infty([0, T_*]; H^2(\mathbb{R}^3))$, and hence converges to a limit function $(\phi, u)$ there strongly. On the
other hand, since the uniform estimates in (\ref{jkk}) is independent of  $k$, we  know that  there exists a subsequence of solutions (still denoted by) $(\phi^k,  u^k)$, which    converges to a limit   $(\phi,  u)$ in the following  weak  or  weak* sense:
\begin{equation}\label{ruojixianqq}
\begin{split}
(\phi^k, u^k)\rightharpoonup  (\phi,u) \quad &\text{weakly* \ in } \ L^\infty([0,T_*];H^3(\mathbb{R}^3)),\\
\phi^k_t\rightharpoonup  \phi_t \quad &\text{weakly* \ in } \ L^\infty([0,T_*];H^2(\mathbb{R}^3)),\\
 u^k_t\rightharpoonup  u_t \quad &\text{weakly* \ in } \ L^\infty([0,T_*];H^1(\mathbb{R}^3)),\\
 u^k_t\rightharpoonup  u_t \quad &\text{weakly \ \  in } \ \ L^2([0,T_*];D^2(\mathbb{R}^3)),\\
\phi^k \nabla^4 u^k \rightharpoonup  \phi\nabla^4 u \quad & \ \text{weakly \ in } \ L^2([0,T_*]\times \mathbb{R}^3),
\end{split}
\end{equation}
which, from the  lower semi-continuity of norm for weak or weak$^*$ convergence, imply that the local estimates in (\ref{jkk}) still hold for  the limit function $(\phi,  u)$.

Now, it is easy to show that
$(\phi,  u)$ is a weak solution of (\ref{eq:cccq})-(\ref{sfanb1}) in  the sense of distribution  with the following regularities:
\begin{equation}\label{rjkqq}\begin{split}
& \phi-\overline{\phi} \in L^\infty([0,T_*];H^3),\quad  \phi_t \in L^\infty([0,T_*];H^2),\\
& u\in L^\infty([0,T_*]; H^3),\quad  \phi \nabla^4 u\in L^2([0,T_*];L^2), \\
& u_t \in L^\infty([0,T_*]; H^1)\cap L^2([0,T_*] ; D^2).
\end{split}
\end{equation}
The existence of  strong solutions is proved.\\

\noindent\underline{Step 2}: Uniqueness.   Let $(\phi_1,u_1)$ and $(\phi_2,u_2)$ be two strong solutions  to Cauchy problem (\ref{eq:cccq})-(\ref{sfanb1})  satisfying the uniform a priori estimates (\ref{jkk}). Denote
\begin{equation*}\begin{split}
&\overline{\varphi}=\phi_1-\phi_2,\quad \overline{u}=u_1-u_2,\\
&\overline{\Phi}=\Phi_1-\Phi_2=\phi^{\frac{2\gamma-\delta-1}{\delta-1}}_1-\phi^{\frac{2\gamma-\delta-1}{\delta-1}}_2.
\end{split}
\end{equation*}
It follows from (\ref{eq:cccq}) that $(\overline{\varphi},\overline{u})$ satisfies the following  system
 \begin{equation}
\label{zhuzhu}
\begin{cases}
\displaystyle
\ \overline{\varphi}_t+u_1\cdot \nabla\overline{\varphi} +\overline{u}\cdot\nabla\phi_{2}+\frac{\delta-1}{2}(\overline{\varphi} \text{div}u_2 +\phi_{1}\text{div}\overline{u})=0,\\[10pt]
\displaystyle
\ \overline{u}_t+ u_1\cdot\nabla \overline{u}+ \overline{u}\cdot \nabla u_{2}+\frac{2A\gamma}{\delta-1} (\Phi_1\nabla \overline{\varphi}+\overline{\Phi}\nabla \phi_2) \\[10pt]
=-(\phi^2_1L\overline{u}+\overline{\varphi}(\phi_1+\phi_2)Lu_2)\\[10pt]
\displaystyle
\quad +\nabla \phi^2_1(Q(u_1)-Q(u_2))+\nabla (\overline{\varphi}(\phi_1+\phi_2)) Q(u_2),
\end{cases}
\end{equation}
with zero initial data.
Let
$$
\Psi(t)=|\overline{\varphi}(t)|^2_{2}+|\overline{u}(t)|^2_2.
$$
Using the same arguments as in the derivation of (\ref{go64})-(\ref{gogo1}), we  have
\begin{equation}\label{gonm}
\begin{cases}
\displaystyle
\frac{d}{dt}\Psi(t)+C|\phi_1\nabla \overline{u}(t)|^2_2\leq F(t)\Psi (t),\\[10pt]
\displaystyle
\int_{0}^{t}F(s)ds\leq C\quad  \text{for} \quad 0\leq t\leq T_*.
\end{cases}
\end{equation}
From Gronwall's inequality, we conclude that $\overline{\varphi}=\overline{u}=0$.
Then the uniqueness is obtained.\\

\noindent\underline{Step 3}: Time-continuity. It can be proved by  the  same arguments as in the proof of Lemma \ref{lem1q}. We omit the details here.

\end{proof}

\subsection{Proof of  Theorem \ref{th2} and Corollary \ref{co2}}
Based on Theorem \ref{ths1}, we are now ready to prove the local existence of regular solution to the original Cauchy problem (\ref{eq:1.1})-(\ref{eq:2.211}). \\

\subsubsection{\bf Proof of Theorem \ref{th2}.}
\begin{proof}
For the initial data (\ref{th78}), we know from Theorem \ref{ths1} that
there exists  a  time $T_* > 0$ such that the problem (\ref{eq:cccq})-(\ref{sfanb1}) admits a unique strong solution $(\phi,u)$ satisfying the regularities in (\ref{reg11qq}), which means that
\begin{equation}\label{reg2}
\rho^{\frac{\delta-1}{2}}=\phi\in C^1(\mathbb{R}^3\times[0, T_*]).
\end{equation}
Noticing that  $\rho=\phi^{\frac{2}{\delta-1}}$, and
$\frac{2}{\delta-1}\geq 1$ for $ 1< \delta \leq  \min \Big(3,\frac{\gamma+1}{2}\Big) $, it is
easy to show that
$$\rho(x,t)\in C^1(\mathbb{R}^3\times[0, T_*]).$$

Now we verify that $(\rho,u)$ satisfies the original equations \eqref{eq:1.1}. Multiplying both sides of $(\ref{eq:cccq})_1$ by
$$
\frac{\partial \rho}{\partial \phi}(x,t)=\frac{2}{\delta-1}\phi^{\frac{3-\delta}{\delta-1}}(x,t)\in C(\mathbb{R}^3\times[0, T_*]),
$$
we get the continuity equation in $(\ref{eq:1.1})_1$.

Multiplying  both sides of $(\ref{eq:cccq})_2$ by
$$
\phi^{\frac{2}{\delta-1}}=\rho(x,t)\in C^1(\mathbb{R}^3\times[0, T_*]),
$$
we get the momentum equations in $(\ref{eq:1.1})_2$.
Therefore, $(\rho,u)$ is a solution to  (\ref{eq:1.1})-(\ref{eq:2.211}) in the  sense of distribution with the regularities shown in Definition \ref{d1}.

Recalling that  $\rho$ can be represented by the formula
$$
\rho(x,t)=\rho_0(U(0;x, t))\exp\Big(\int_{0}^{t}\textrm{div} u(s,U(s;x,t))\text{d}s\Big),
$$
where  $U\in C^1([0,T_*]\times\mathbb{R}^3\times[0, T_*])$ is the solution to the initial value problem
\begin{equation}
\label{eq:bb1}
\begin{cases}
\frac{d}{ds}U(s;x,t)=u(s,U(s;x,t)),\quad 0\leq s\leq T_*,\\[4pt]
U(t;x,t)=x, \quad\quad \quad  x\in \mathbb{R}^3, \quad 0\leq t\leq T_*,
\end{cases}
\end{equation}
 it is obvious that
$$
\rho(x,t)\geq 0, \ \forall (x,t)\in \mathbb{R}^3\times[0, T_*].
$$
In summary, the Cauchy problem (\ref{eq:1.1})-(\ref{eq:2.211}) has a unique regular solution $(\rho,u)$.
\end{proof}

\subsubsection{\bf Proof of Corollary \ref{co2}}.
\begin{proof} First, from  $1< \delta \leq \frac{5}{3}$  we  know that  $\frac{2}{\delta-1}\geq 3$. Since
$$\phi-\overline{\phi} \in C([0,T_*];H^{3})\cap C^1([0,T_*]; H^{2}),$$ and
$$
\rho(x,t)=\phi^{\frac{2}{\delta-1}}(x,t),
$$
we have
$$
\rho-\overline{\rho} \in C([0,T_*];H^{3}).
$$

Second,  due to  the fact that
\begin{equation}\label{ze1}
\begin{split}
&  \rho^{\frac{\delta-1}{2}} -\overline{\rho}^{\frac{\delta-1}{2}}\in C([0,T_*]; H^3),\ \ u\in C([0,T_*]; H^{s'}) \cap L^\infty([0,T_*]; H^3),\\
&\rho^{\frac{\delta-1}{2}} \nabla^4 u \in L^2([0,T_*]; L^2),\ \ u_t \in C([0,T]; H^1)\cap L^2([0,T] ; D^2),
\end{split}
\end{equation}
for any constant  $s' \in[2,3)$, and the same  arguments as   in Lemma \ref{lem1q} for the time continuity,  we deduce that
\begin{equation}\label{ze2}
\rho \text{div}u \in L^2([0,T_*]; H^3) \cap C([0,T_*]; H^2).
\end{equation}

At last,
with the aid of the continuity equation
$\rho_t+u \cdot\nabla \rho+\rho\text{div} u=0$
and (\ref{ze1})-(\ref{ze2}),   it is clear that $$
\rho-\overline{\rho} \in C([0,T_*];H^{3})\cap C^1([0,T_*];H^{2}).
$$

Furthermore, when  $\delta=2$ and  $\gamma\geq 3$,  by the same token, the regularity of $\rho$ in these cases can be achieved. \end{proof}

\section{Formation of singularities}

In this section, we consider the formation of singularities of regular solutions obtained in Section $3$. Two
classes of  initial data that lead to finite time blow-up  will be given. We assume that  $(\rho, u)(x,t)$ is the regular solution in  $\mathbb{R}^3\times[0, T_m)$ obtained in Theorem \ref{th2}, with $T_m$ the maximal existence time.

\subsection{Blow-up by isolated mass group} The first kind of singularity formation is
driven by isolated mass group, defined in Definition \ref{local}. Assume that the initial data $(\rho_0, u_0)$ have an isolated mass group $(A_0, B_0)$, the following definition helps to track the evolution of $A_0$ and $B_0$.
\begin{definition}[\textbf{Particle path and flow map}]\label{kobn}\ Let $x(t;x_0)$ be the particle path starting from $x_0$ at $t=0$, i.e.,
\begin{equation}\label{gobn}
\frac{d}{\text{d}t}x(t;x_0)=u(x(t;x_0),t),\quad x(0;x_0)=x_0.
\end{equation}
Let $A(t)$, $B(t)$, $B(t)\setminus A(t)$ be the images of $A_0$, $B_0$, and $B_0 \setminus A_0$, respectively, under the flow map of (\ref{gobn}), i.e.,
\begin{equation*}
\begin{split}
&A(t)=\left\{x(t;x_0)|x_0\in A_0\right\},\\[2pt]
& B(t)=\left\{x(t;x_0)|x_0\in B_0\right\},\\[2pt]
&B(t)\setminus A(t)=\left\{x(t;x_0)|x_0\in B_0\setminus  A_0\right\}.
\end{split}
\end{equation*}
\end{definition}

The following lemma confirms the invariance of  the volume $|A(t)|$ for regular solutions.
\begin{lemma}
\label{lemma:3.1}
Suppose that  the initial data $(\rho_0,u_0)(x)$ have an isolated mass group $(A_0,B_0)$,
then for the   regular solution $(\rho,u)(x,t)$ on $\mathbb{R}^3\times[0,T_m)$ to the Cauchy problem (\ref{eq:1.1})-(\ref{eq:2.211}), we have
$$
|A(t)|=|A_0|, \quad  t\in [0,T_m).
$$
\end{lemma}

\begin{proof}
Since
$$
\rho(x(t;x_0),t)=\rho_0(x_0)\exp\Big(\int_{0}^{t}\textrm{div} u(x(s;x_0), s)\text{d}s\Big),
$$
it is clear that
$$ \rho \equiv 0, \quad \text{in } \quad B(t)\setminus A(t).$$
From    the definition of  regular solutions, we have
\begin{equation}\label{eq:5.3}
u_t+u\cdot\nabla u=0, \quad \text{in} \quad   B(t)\setminus A(t).
\end{equation}
Therefore, $u$ is invariant along the particle path $x(t;x_0)$ with $x_0\in B_0 \setminus A_0$.

For any $x^1_0,\ x^2_0 \in \partial A_0$, we define
\begin{equation}\label{gobn1}
\frac{d}{\text{d}t}x^i(t;x^i_0)=u(x^i(t;x^i_0), t),\quad x^i(0;x^i_0)=x^i_0,\quad \text{for} \quad i=1,2.
\end{equation}
Then we  have
\begin{equation}\label{gobn2}
\frac{d}{\text{d}t}(x^1(t;x^1_0)-x^2(t;x^2_0))=u(x^1(t;x^1_0), t)-u(x^2(t;x^2_0), t)=0,
\end{equation}
which implies that
$$
|A(t)|=|A_0|, \quad  t\in [0,T_m].
$$
\end{proof}

We point out that, although the volume of $A(t)$ is invariant, the vacuum boundary $\partial A(t)$ can vary as time evolves. In order to deal with this situation,  the  following well-known Reynolds transport theorem (c.f. \cite{kong}) is useful.
\begin{lemma}\label{3.1}
For any $G(x,t)\in C^1(\mathbb{R}^3\times\mathbb{R}^+) $, one has
$$
\frac{d}{\text{d}t}\int_{A(t)}  G(x,t)\text{d}x= \int_{A(t)}  G_t(x,t)\text{d}x+\int_{\partial A(t)} G(x,t)(u(x,t)\cdot {\vec n})\text{d}S,
$$
where ${\vec n}$ is the outward unit normal vector to $\partial A(t)$, and $u$ is the velocity of the fluid.
\end{lemma}

In the rest part of this section, we will use the following useful physical  quantities on the fluids in $A(t)$:
\begin{align*}
&m(t)=\int_{A(t)}\rho(x,t)\text{d}x \quad \textrm{(total mass)},\\
&M(t)=\int_{A(t)} \rho(x,t)|x|^{2}\text{d}x \quad \textrm{ (second moment)},\\
&F(t)=\int_{A(t)} \rho(x,t)u(x,t)\cdot x \text{d}x \quad \textrm{ (radial component of momentum)},\\
&\varepsilon(t)=\int_{A(t)} \Big(\frac{1}{2}\rho |u|^2+\frac{P}{\gamma-1}\Big)(x,t) \text{d}x \quad  \textrm{(total  energy)}.
\end{align*}

From the continuity equation, it is clear that the mass is conserved.
\begin{lemma}\label{3.2}
 Suppose that  the initial data $(\rho_0,u_0)(x)$ have  an isolated mass group $(A_0,B_0)$,
then for the   regular solution $(\rho,u)(x, t)$ on $\mathbb{R}^3\times  [0,T_m)$ to the Cauchy problem (\ref{eq:1.1})-(\ref{eq:2.211}), we have
$$m(t)=m(0),\quad \text{for} \quad t\in [0,T_m).
$$
\end{lemma}
\begin{proof} From $(\ref{eq:1.1})_1$ and Lemma \ref{3.1}, direct computation shows
\begin{equation*}\begin{split}
\frac{d}{dt}m(t)=&\int_{A(t)} \rho_t\  \text{d}x+\int_{\partial A(t)} \rho u\cdot {\vec n}\ \text{d}S\\
=&\int_{A(t)} -\text{div}(\rho u)\  \text{d}x=\int_{\partial A(t)} -\rho u\cdot {\vec n}\ \text{d}S=0,
\end{split}
\end{equation*}
which implies that $m(t)=m(0)$.

\end{proof}

Motivated by \cite{zx},  we define the following functional
\begin{equation} \label{cd:3}
\begin{split}
I(t)&=M(t)-2(t+1)F(t)+2(t+1)^{2}\varepsilon(t)\\
&=\int_{A(t)} |x-(t+1)u|^{2}\rho\text{d}x +\frac{2}{\gamma-1}(t+1)^{2}\int_{A(t)} P\text{d}x.
\end{split}
\end{equation}

We now follow the arguments of \cite{zx} with some proper modifications to prove Theorem \ref{coo2}.
One of the key observations in the following proof is  that the viscosity tensor $\mathbb{T }=0$ in vacuum region
due to \eqref{eq:1.3}.

\begin{proof}[{\bf Proof of Theorem 1.2}] From system (\ref{eq:1.1}), it is clear that
\begin{equation} \label{eq:5.8}
\begin{split}
 \left( \frac{1}{2}\rho|u|^2+\frac{P}{\gamma-1}\right)_t&=-\text{div}\left(\frac{1}{2}\rho|u|^2u\right)-\frac{\gamma}{\gamma-1}\text{div} (P u)+u\cdot \text{div}\mathbb{T }.
\end{split}
\end{equation}
From the continuity equation $(\ref{eq:1.1})_1$, momentum equations $(\ref{eq:1.1})_2$, relation (\ref{eq:5.8}), Lemma \ref{3.1} and integration by parts, we have
\begin{equation} \label{eq:5.6}
\begin{split}
\frac{d}{\text{d}t}I(t)=&\frac{d}{dt}M(t)-2(t+1)\frac{d}{dt}F(t)+2(t+1)^{2}\frac{d}{dt}\varepsilon(t)-2F(t)+4(t+1)\varepsilon(t)\\
\displaystyle
=&\frac{2}{\gamma-1}(2-3(\gamma-1))(t+1)\int_{A(t)} P \text{d}x+\overline{J}_1+\overline{J}_2,
\end{split}
\end{equation}
where  $\overline{J}_1$ and $\overline{J}_2$ are given by
\begin{equation*}
\overline{J}_1=-2(t+1)\int_{A(t)} x \cdot \text{div}\mathbb{T } \text{d}x,\ \ \overline{J}_2=2(t+1)^2 \int_{A(t)} u\cdot \text{div}\mathbb{T }\text{d}x.
\end{equation*}

Since
\begin{equation}\label{haoq}
\begin{split}
\text{div} (x\cdot \mathbb{T })
=x\cdot (\text{div}\mathbb{T }) +\sum_{i=1}^{3}\mathbb{T }_{ii}=x \cdot (\text{div} \mathbb{T})+3\left(\beta+\frac{2}{3} \alpha\right)\rho^\delta\text{div}u.
\end{split}
\end{equation}
Integrating (\ref{haoq}) by parts, with help of the fact $\mathbb{T }=0$ on $\partial A(t)$, we have
\begin{equation} \label{eq:3.511}
\overline{J}_1=-2(t+1)\int_{A(t)} x \cdot \text{div} \mathbb{T}\text{d}x=6(t+1)\int_{A(t)}\left(\beta+\frac{2}{3}\alpha\right) \rho^\delta \text{div}u\text{d}x.\end{equation}

Now we turn to $\overline{J}_2$.
From (\ref{fandan}) and Cauchy's inequality we have
\begin{equation}\label{eq:5.9}
\begin{split}
 \text{div} (u \mathbb{T })
=&u\cdot \text{div}\mathbb{T }+2\mu(\rho)\sum_{i=1}^{3}(\partial_{i}u_{i})^{2}+\mu(\rho)  \sum_{i\neq j}^{3} (\partial_{i}u_{j})^2\\
&\quad+2\mu(\rho) \sum_{i>j} (\partial_{i}u_{j})(\partial_{j}u_{i})+\lambda(\rho) \left(\sum_{i=1}^{3}\partial_{i}u_{i}\right)^{2}\\
\geq& u\cdot \text{div}\mathbb{T }+\left(\lambda(\rho)+\frac{2}{3}\mu(\rho)\right)(\text{div} u)^2\\
=&u\cdot \text{div}\mathbb{T }+\left(\beta+\frac{2}{3}\alpha\right)\rho^\delta (\text{div} u)^2.
\end{split}
\end{equation}
We integrate (\ref{eq:5.9}) over $A(t)$ to get
\begin{equation}\label{eq:5.10}
\int_{A(t)} u\cdot \text{div}\mathbb{T }\text{d}x\leq -\int_{A(t)}\left(\beta+\frac{2}{3}\alpha\right)\rho^\delta (\text{div} u)^2\text{d}x,
\end{equation}
from which we have
\begin{equation}\label{k9}
\overline{J}_2\leq -2(t+1)^2\int_{A(t)}\left(\beta+\frac{2}{3}\alpha\right)\rho^\delta (\text{div} u)^2\text{d}x.
\end{equation}

From Lemma \ref{lemma:3.1}, along with (\ref{eq:5.6}), (\ref{eq:3.511}) and (\ref{k9}), for $0\leq t\leq  T_m$, we get
\begin{equation}\label{eq:5.13}
\begin{split}
\frac{d}{\text{d}t}I(t)\leq & \frac{2}{\gamma-1}(2-3(\gamma-1))(t+1)\int_{A(t)}P\text{d}x
\\
&-2(t+1)^2\int_{A(t)}\left(\beta+\frac{2}{3}\alpha\right)\rho^\delta(\text{div} u)^2\text{d}x\\
&+6(t+1)\int_{A(t)} \left(\beta+\frac{2}{3}\alpha\right)\rho^\delta \text{div}u\text{d}x.
\end{split}
\end{equation}
Since $\delta\le \gamma$, with the help of Cauchy's inequality and Young's inequality, we have
\begin{equation} \label{11111}
\begin{split}
&-2(t+1)^2\int_{A(t)}\rho^\delta(\text{div} u)^2\text{d}x+6(t+1)\int_{A(t)} \rho^\delta \text{div}u\text{d}x \\
\leq & -2(t+1)^2\int_{A(t)}\rho^\delta(\text{div} u)^2\text{d}x+2(t+1)^2\int_{A(t)}\rho^\delta(\text{div} u)^2\text{d}x+18\int_{A(t)}\rho^\delta \text{d}x\\
\leq & 18\int_{A(t)}\rho^\delta  \text{d}x\leq \frac{18\delta}{\gamma}\int_{A(t)}\rho^\gamma \text{d}x+\frac{18(\gamma-\delta)}{\gamma}|A_{0}|.
\end{split}
\end{equation}
We deduce from (\ref{eq:5.13}) that
\begin{equation}\label{eq:3.1331}
\begin{split}
\frac{d}{\text{d}t}I(t) \leq & \frac{2}{\gamma-1}(2-3(\gamma-1))(t+1)\int_{A(t)}P\text{d}x \\
&+18\left(\beta+\frac{2}{3}\alpha\right)\frac{\delta}{\gamma}\int_{A(t)}\rho^\gamma \text{d}x+18\left(\beta+\frac{2}{3}\alpha\right)\frac{\gamma-\delta}{\gamma}|A_{0}|.
\end{split}
\end{equation}
From the second  expression of $I(t)$ in (\ref{cd:3}), one has
\begin{equation}\label{eq:3.131}
\begin{split}
\frac{2-3(\gamma-1)}{t+1}I(t)=&\frac{2-3(\gamma-1)}{t+1}\int_{A(t)} |x-(t+1)u|^{2}\rho\text{d}x \\
&+\frac{2}{\gamma-1}(2-3(\gamma-1))(t+1)\int_{A(t)} P\text{d}x.
\end{split}
\end{equation}

In the case when $1< \gamma < \frac{5}{3}$, from  (\ref{eq:3.1331})-(\ref{eq:3.131}), for $0\leq  t<  T_m$, we have
\begin{equation}\label{eq:3.141}
\frac{d}{\text{d}t}I(t) \leq \frac{2-3(\gamma-1)}{t+1} I(t)+18\left(\beta+\frac{2}{3}\alpha\right)\frac{\delta(\gamma-1)}{2A\gamma(t+1)^2}I(t)+18\left(\beta+\frac{2}{3}\alpha\right)\frac{\gamma-\delta}{\gamma}|A_{0}|.
\end{equation}
Solving (\ref{eq:3.141}) directly, we get
\begin{equation}\label{k1}
\begin{split}
I(t) \leq & (t+1)^{2-3(\gamma-1)} e^{-\frac{a_1}{t+1}} \left( e^{a_1} I(0)+a_2 \int_{0}^{t}(\tau+1)^{3(\gamma-1)-2} e^{  \frac{a_1}{\tau+1} } \text{d} \tau \right),
\end{split}
\end{equation}
where
$$
a_1=18\left(\beta+\frac{2}{3}\alpha\right)\frac{\delta(\gamma-1)}{2A\gamma},\quad  a_2=18\left(\beta+\frac{2}{3}\alpha\right)\frac{\gamma-\delta}{\gamma}|A_{0}|.
$$
If $3(\gamma-1)-2 \neq -1$, then from (\ref{k1}) we get
\begin{equation}\label{eq:3.ode1}
\begin{split}
I(t)
\leq & (t+1)^{2-3(\gamma-1)}e^{-\frac{a_1}{t+1}}\left(e^{a_1}I(0)-\frac{a_2e^{a_1}}{3(\gamma-1)-1}\right)
 +\frac{a_2(t+1)}{ 3(\gamma-1)-1} e^{-\frac{a_1}{t+1}} e^{a_1}\\
\displaystyle
\leq & C\left(t^{2-3(\gamma-1)}+ t+1\right), \quad  \text{for} \ t\in [0,T_m).
\end{split}
\end{equation}
If $3(\gamma-1)-2=-1$, from (\ref{k1}) we get
\begin{equation}\label{k2}
\begin{split}
I(t) \leq & (t+1)^{2-3(\gamma-1)}e^{-\frac{a_1}{t+1}} \left( e^{a_1} I(0)+a_2e^{a_1}\ln(t+1) \right)\\
\displaystyle
 \leq & C\left((t+1)\ln(t+1)+t+1\right), \quad \text{for} \ t\in [0,T_m).
\end{split}
\end{equation}
On the other hand, from  the definition of $I(t)$, Jensen's inequality and Lemma \ref{lemma:3.1}, we show that
\begin{equation}\label{ode:3.21}
\begin{split}
I(t)
\geq&  \frac{2(t+1)^2}{\gamma-1} |A_{0}|\int_{A(t)} A\rho^\gamma(x,t)\frac{\text{d}x}{|A(t)|} \\
\geq & \frac{C(t+1)^2}{\gamma-1} |A_{0}|^{1-\gamma} m(0)^\gamma
\geq C_0 (1+t)^2 ,
\end{split}
\end{equation}
where $C_0>0$ is a constant and  we used the fact in Lemma \ref{3.2} that
$$m(t)= \int_{A(t)} \rho(x,t) \text{d}x =\int_{A_0}\rho_0(x)\text{d}x=m(0).$$
Then $T_m<+\infty$ follows immediately, otherwise a contradiction forms between (\ref{ode:3.21}) and (\ref{eq:3.ode1}) or \eqref{k2}.

In the case when $ \frac{5}{3}\leq \gamma < +\infty $, thus $2-3(\gamma-1)\leq 0$, from (\ref{eq:3.1331}) we have
\begin{equation}\label{eq:3.1335}
\begin{split}
\frac{d}{\text{d}t}I(t) \leq & 18\left(\beta+\frac{2}{3}\alpha\right) \frac{\delta}{\gamma}\int_{A(t)}\rho^\gamma \text{d}x+18\left(\beta+\frac{2}{3}\alpha\right)\frac{(\gamma-\delta)}{\gamma}|A_{0}|\\
\leq &18\left(\beta+\frac{2}{3}\alpha\right)\frac{\delta(\gamma-1)}{2\gamma A (t+1)^2}I(t)+18\left(\beta+\frac{2}{3}\alpha\right)\frac{\gamma-\delta}{\gamma}|A_{0}|\\
&=\frac{a_1}{(t+1)^2}I(t)+a_2.
\end{split}
\end{equation}

and therefore,
\begin{equation}
\begin{split}
I(t)&\le e^{a_1}I(0)e^{-\frac{a_1}{t+1}}+a_2e^{a_1} e^{-\frac{a_1}{t+1}}e^{a_1}t\\
    &\le e^{a_1}(I(0)+a_2t).
\end{split}
\end{equation}

Again, this and \eqref{ode:3.21} imply that $T_m<\infty$. We complete the proof of Theorem 1.2.
\end{proof}

\subsection{Blow-up by hyperbolic singularity set}

The mechanism for our second finite time blow-up result comes from the nonlinear hyperbolic structure (see (\ref{zhenkong})) which controls the behavior of the velocity $u$ in the vacuum region. Assume that the initial data $(\rho_0, u_0)(x)$ have a hyperbolic singularity set $V$, see Definition \ref{bugers}.
\begin{proof}[{\bf Proof of Theorem 1.3}]
Let $ V(t) $ be the image of $ V$ under the flow map of (\ref{gobn}), i.e.,
\begin{equation}\label{zhi}
V(t)=\{x(t;\xi_0)| \xi_0 \in V\}.
\end{equation}

      It follows from the continuity equation $(\ref{eq:1.1})_1$ that the  density is simply transported along the particle path, so
       $$ \rho(x,t)=0,\quad \text{when} \quad x\ \in \ V(t).$$
From the Definition \ref{d1} for regular solutions, we have
\begin{equation}
\label{eq:1.2guogai}
u_t+u\cdot \nabla u=0, \quad \text{when} \quad x\ \in \ V(t),
\end{equation}
which means that $u$ is a constant  vector along the particle path $x(t; \xi_0)$ and
$$
\xi_0=x-tu(x,t)\in V.
$$
Then for any $x\in V(t)$, we have
$$
u(x,t)=u_0(\xi_0)=u_0(x-tu(x,t)),
$$
which  implies that
\begin{equation}
\label{eq:1.2fan}
\begin{split}
\nabla u(x,t)=&\big(\mathbb{I}_3+t\nabla u_0(x-tu(x,t))\big)^{-1}\nabla u_0(x-tu(x,t))\\
=&\big(\mathbb{I}_3+t\nabla u_0(\xi_0)\big)^{-1}\nabla u_0(\xi_0),\quad \text{for} \quad x\in V(t).
\end{split}
\end{equation}
According to the definition of the hyperbolic singularity set, there exists some
$$\xi_0 \in V,\quad \text{and} \quad  l_{\xi_0} \in  Sp(\nabla u_0 (\xi_0)) \quad  \text{satisfying} \quad  l_{\xi_0}<0.$$ Let $w \in \mathbb{R}^3$ be the eigenvector of $\nabla u_0(\xi_0)$ with respect to  $l_{\xi_0}$, that is,
$$
\nabla u_0(\xi_0)w= l_{\xi_0} w.
$$
It is clear that
$$
\big(\mathbb{I}_3+t\nabla u_0(\xi_0)\big)^{-1}w=\big(1+t l_{\xi_0}\big)^{-1} w.
$$
Thus we know that  the matrix $\nabla u(x,t)$  has an eigenvector $w$ with the eigenvalue
$$
\frac{l_{\xi_0}}{1+tl_{\xi_0}},
$$
which, along with $l_{\xi_0}<0$,   implies   that the quantity $\nabla u$ will blow up in finite time, i.e.,
$$
T_m<+\infty.
$$
This completes the proof of Theorem 1.3.
\end{proof}

\section{Appendix: {Proof for the Remark \ref{re3}}  }

In this section,  we will show that the regular solution that we obtained in Theorem \ref{th2} is indeed a classical one in $(0, T_*]$.  The following lemma will be used in our proof.
\begin{lemma}\cite{bjr}\label{1}
If $f(x,t)\in L^2([0,T]; L^2)$, then there exists a sequence $s_k$ such that
$$
s_k\rightarrow 0, \quad \text{and}\quad s_k |f(x,s_k)|^2_2\rightarrow 0, \quad \text{as} \quad k\rightarrow+\infty.
$$
\end{lemma}

From  the definition of regular solution and  the classical Sobolev embedding theorem, it is clear  that
$$
(\rho,\nabla \rho, \rho_t, u, \nabla u) \in C(\mathbb{R}^3\times[0, T_*]).
$$
So it remains  to prove that
$$
(u_t, \text{div}\mathbb{T})(x,t) \in C(\mathbb{R}^3\times (0, T_*]).
$$

From  the proof of Theorem \ref{th2} in Section $3$, we know that (through a change of variable $\phi=\rho^{\frac{\delta-1}{2}}$),   system (\ref{eq:1.1})  can be written as
\begin{equation}
\begin{cases}
\label{rfv}
\displaystyle
\phi_t+u\cdot \nabla \phi+\frac{\delta-1}{2}\phi \text{div} u=0,\\[10pt]
\displaystyle
u_t+u\cdot\nabla u +\frac{2A\gamma}{\delta-1}\phi^{\frac{2r-\delta-1}{\delta-1}}\nabla \phi+\phi^2 Lu=\nabla \phi^2 \cdot Q(u).
 \end{cases}
\end{equation}
The solution $(\phi,u)$ satisfies the regularities in (\ref{reg11qq}) and $\phi \in C^1(\mathbb{R}^3\times[0, T_*])$. \\

\noindent\underline{Step 1}: The continuity of $u_t$.
We differentiate $(\ref{rfv})_2$ with respect to $t$ to get
\begin{equation}\label{zhd1}
\begin{split}
u_{tt}+\phi^2 Lu_t=-(\phi^2)_t Lu-(u\cdot\nabla u)_t -\frac{A\gamma}{\gamma-1}\nabla \Big(\phi^{\frac{2\gamma-2}{\delta-1}}\Big)_t+(\nabla \phi^2 \cdot Q(u))_t,
\end{split}
\end{equation}
which, along with (\ref{reg11qq}),  implies that
\begin{equation}\label{zhd11}
u_{tt}\in L^2([0,T_*];L^2).
\end{equation}
Applying the operator $\partial^\zeta_x$  $(|\zeta|=2)$ to $(\ref{zhd1})$,  multiplying the resulting equations by $\partial^\zeta_x u_t$ and integrating over $\mathbb{R}^3$, we have
\begin{equation}\label{zhd2}
\begin{split}
&\frac{1}{2} \frac{d}{dt}|\partial^\zeta_xu_t|^2_2+\alpha|\phi \nabla \partial^\zeta_x u_t|^2_2+(\alpha+\beta)|\phi \text{div} \partial^\zeta_x u_t|^2_2\\
=&\int  \Big( -\nabla \phi^2 \cdot \frac{\delta-1}{\delta}Q(\partial^\zeta_x u_t)-\big(\partial^\zeta_x(\phi^2Lu_t)-\phi^2 L\partial^\zeta_x u_t\big)\Big)\cdot   \partial^\zeta_x u_t\\
&+\int  \Big(-\partial^\zeta_x\big((\phi^2)_t Lu\big)-\partial^\zeta_x(u\cdot\nabla u)_t -\frac{A\gamma}{\gamma-1}\partial^\zeta_x\nabla \Big(\phi^{\frac{2\gamma-2}{\delta-1}}\Big)_t \Big)\cdot   \partial^\zeta_x u_t \\
&+\int  \partial^\zeta_x(\nabla \phi^2 \cdot Q(u))_t\cdot   \partial^\zeta_x u_t \\
 =: &\sum_{i=10}^{15}J_i.
\end{split}
\end{equation}
Now we analyze the terms $J_i$ $(i=10,\cdots, 15)$. By H\"older's inequality, Lemma \ref{lem2as} and Young's inequality, we have
\begin{equation}\label{zhd3}
\begin{split}
J_{10}=&\int  \Big( -\nabla \phi^2 \cdot \frac{\delta-1}{\delta}Q(\partial^\zeta_x u_t)\Big)\cdot   \partial^\zeta_x u_t\\
\leq & C|\phi\nabla^3 u_t|_2|\nabla^2 u_t|_2|\nabla \phi|_\infty\leq C| u_t|^2_{D^2}+\frac{\alpha}{20}|\phi\nabla^3 u_t|^2_2,\\
J_{11}=&\int -\big(\partial^\zeta_x(\phi^2Lu_t)-\phi^2 L\partial^\zeta_x u_t\big)\cdot   \partial^\zeta_x u_t\\
\leq & C\Big(|\phi\nabla^3 u_t|_2|\nabla^2 u_t|_2|\nabla \phi|_\infty+|\nabla \phi|^2_\infty|u_t|^2_{D^2}+|\nabla^2 \phi|_3|\phi \nabla^2 u_t|_6|u_t|_{D^2}\Big)\\
\leq & C| u_t|^2_{D^2}+\frac{\alpha}{20}|\phi\nabla^3 u_t|^2_2,\\
J_{12}=&\int  -\partial^\zeta_x\big((\phi^2)_t Lu\big)\cdot   \partial^\zeta_x u_t \\
\leq &C\Big(|\phi_t|_\infty|u_t|_{D^2} |\phi \nabla^4 u|_2+|\nabla^2 \phi|_3|Lu|_6| \phi_t|_\infty|u_t|_{D^2}\\
&\quad+|\phi \nabla^2 u_t|_6|\phi_t|_{D^2}|Lu|_3+|\nabla \phi|_\infty|\nabla \phi_t|_6|Lu|_{3}|u_t|_{D^2}\\
&\quad+|\phi \nabla^3 u|_6|\nabla \phi_t|_{3}|u_t|_{D^2}+|\nabla \phi|_\infty| \phi_t|_\infty|\nabla^3 u|_{2}|u_t|_{D^2}\Big)\\
\leq &C\Big(| u_t|^2_{D^2}+|\phi \nabla^4 u|^2_2+1\Big)+ \frac{\alpha}{20}|\phi\nabla^3 u_t|^2_2,\\
J_{13}=&\int  -\partial^\zeta_x(u\cdot\nabla u)_t \cdot   \partial^\zeta_x u_t \\
\leq & C\|u_t\|_2\|u\|_3| u_t|_{D^2}+\int -\big(u\cdot \nabla\big) \partial^\zeta_x u_t \cdot   \partial^\zeta_x u_t \\
\leq &  C\Big(1+| u_t|^2_{D^2}+|\nabla u|_\infty|\partial^\zeta_x u_t|^2_2\Big) \leq   C\Big(1+| u_t|^2_{D^2}\Big),
\end{split}
\end{equation}
\begin{equation}
\begin{split}
J_{14}=&\int   -\frac{A\gamma}{\gamma-1}\partial^\zeta_x\nabla \Big(\phi^{\frac{2\gamma-2}{\delta-1}}\Big)_t \cdot   \partial^\zeta_x u_t=\int   \frac{A\gamma}{\gamma-1}\partial^\zeta_x \Big(\phi^{\frac{2\gamma-2}{\delta-1}}\Big)_t \text{div}  \partial^\zeta_x u_t\\
\leq& C\Big(|\phi^{\frac{K}{2}-2}|_\infty|\nabla^2 \phi_t|_2|\phi \nabla^3 u_t|_2+|\phi^{\frac{K}{2}-3}|_\infty |\phi_t|_\infty |\nabla^2 \phi|_2|\phi \nabla^3 u_t|_2\\
&\quad+|\phi^{\frac{K}{2}-4}|_\infty|\phi_t|_\infty|\nabla \phi|_6|\nabla \phi|_3|\phi \nabla^3 u_t|_2+|\phi^{\frac{K}{2}-3}|_\infty|\nabla \phi_t|_2|\nabla \phi|_\infty|\phi \nabla^3 u_t|_2\Big)\\
\leq & C+\frac{\alpha}{20}|\phi\nabla^3 u_t|^2_2,
\end{split}
\end{equation}
and
\begin{equation}\label{zhd3jk}
\begin{split}
J_{15}=&\int  \partial^\zeta_x(\nabla \phi^2 \cdot Q(u))_t\cdot   \partial^\zeta_x u_t\\
\leq &C\Big(\|\nabla \phi\|^2_2| u_t|^2_{D^2}+\big(\|\nabla \phi\|_2|\nabla u_t|_3+\|u\|_3\|\phi_t\|_2\big)|\phi \nabla^2 u_t|_6\\
&+\big(\|\nabla \phi\|_2|\phi \nabla^3 u_t|_2+\|\nabla \phi\|_2|\phi \nabla^2 u_t|_6\big)| u_t|_{D^2}\\
&+\big(\|\nabla \phi\|_2\|\phi_t\|_2\|u\|_3+\|\phi_t\|_2|\phi \nabla^3 u|_6\big)| u_t|_{D^2}\Big)+\int  \partial^\zeta_x(\nabla \phi^2)_t \cdot Q(u)\cdot   \partial^\zeta_x u_t\\
\leq &C\Big(| u_t|^2_{D^2}+|\phi\nabla^4 u|^2_2\Big)+ \frac{\alpha}{20}|\phi\nabla^3 u_t|^2_2+J_{151},
\end{split}
\end{equation}
where
\begin{equation}\label{zhd15}
\begin{split}
J_{151}=&\int  \partial^\zeta_x(\nabla \phi^2)_t \cdot Q(u)\cdot   \partial^\zeta_x u_t\\
\leq& C \|\nabla \phi\|_2\|\phi_t\|_2\|u\|_3|u_t|_{D^2}+\int  \phi \partial^\zeta_x \nabla \phi_t \cdot Q(u)\cdot   \partial^\zeta_x u_t.
\end{split}
\end{equation}
Using integration by parts, the last term in (\ref{zhd15}) is estimated as:
\begin{equation}\label{zhd16}
\begin{split}
&\int  \phi \partial^\zeta_x \nabla \phi_t \cdot Q(u)\cdot   \partial^\zeta_x u_t\\
\leq&  C\Big(|\nabla \phi|_\infty|\phi_t|_{D^2}|\nabla u|_\infty|u_t|_{D^2}+|\phi_t|_{D^2}|\nabla^2u|_{3}|\phi \nabla^2 u_t|_6+|\phi\nabla^3 u_t|_2|\nabla u|_\infty\|\phi_t\|_2\Big)\\
\leq & C| u_t|^2_{D^2}+\frac{\alpha}{20}|\phi\nabla^3 u_t|^2_2.
\end{split}
\end{equation}
Then (\ref{zhd2}) reduces to
\begin{equation}\label{zhd5}
\begin{split}
&\frac{1}{2} \frac{d}{dt}|u_t|^2_{D^2}+\frac{\alpha}{2} |\phi \nabla^3 u_t|^2_2
 \leq C\Big(| u_t|^2_{D^2}+|\phi \nabla^4 u|^2_2+1\Big).
\end{split}
\end{equation}
Multiplying both sides of (\ref{zhd5}) with $s$ and integrating the resulting inequalities over $[\tau,t]$ for any $\tau \in (0,t)$, we have
\begin{equation}\label{zhd6}
\begin{split}
&t|u_t|^2_{D^2}+\int_{\tau}^t s|\phi \nabla^3 u_t|^2_2 \text{d}s
 \leq C\tau| u_t(\tau)|^2_{D^2}+C(1+t).
\end{split}
\end{equation}
According to the definition of the regular solution, we know that
$$
\nabla^2 u_t \in L^2([0,T_*]; L^2).
$$
Using Lemma \ref{1} to $\nabla^2 u_t $, there exists a sequence $s_k$ such that
$$
s_k\rightarrow 0, \quad \text{and}\quad s_k |\nabla^2 u_t(\cdot, s_k)|^2_2\rightarrow 0, \quad \text{as} \quad k\rightarrow+\infty.
$$
Choosing $\tau=s_k \rightarrow 0$ in (\ref{zhd6}), we have
\begin{equation}\label{zhd7}
\begin{split}
&t|u_t|^2_{D^2}+\int_{0}^t s|\phi \nabla^3 u_t|^2_2 \text{d}s
 \leq C(1+t),
\end{split}
\end{equation}
then
\begin{equation}\label{zhd12}
t^{\frac{1}{2}}u_t \in L^\infty([0,T_*]; H^2).
\end{equation}
The classical  Sobolev embedding theorem gives
 \begin{equation}\label{yhn}
\begin{split}
L^\infty([0,T];H^1)\cap W^{1,2}([0,T];H^{-1})\hookrightarrow C([0,T];L^q),
\end{split}
\end{equation}
for any $q\in (3,6]$.  From (\ref{zhd11}) and (\ref{zhd12}) we have
$$
tu_t \in C([0,T_*];W^{1,4}),
$$
which implies that
$$
u_t \in C(\mathbb{R}^3\times (0,T_*] ).
$$
\\
\underline{Step 2}: The continuity of $\text{div}\mathbb{T}$.  Denote  $\mathbb{N}=\phi^2 Lu-\nabla \phi^2 \cdot Q(u)$.
From equations $(\ref{rfv})_2$, regularities (\ref{reg11qq}) and (\ref{zhd12}), it is easy to show that
$$t\mathbb{N} \in L^\infty([0,T_*]; H^2).$$
Due to
$$
\mathbb{N}_t \in L^2([0,T_*]; L^2),
$$
 we obtain from (\ref{yhn})  that
$$
t \mathbb{N}\in C([0,T_*];W^{1,4}),
$$
which implies that
$$
\mathbb{N} \in C(\mathbb{R}^3\times (0, T_*]).
$$
Since $\rho \in C(\mathbb{R}^3\times[0, T_*])$ and $\text{div}\mathbb{T}=\rho \mathbb{N}$, we immediately obtain the desired conclusion.

In summary,  we have shown that the regular solution that we obtained is indeed a classical one in $\mathbb{R}^3\times[0, T_*]$ to  Cauchy problem (\ref{eq:1.1})-(\ref{eq:2.211}).

\bigskip

{\bf Acknowledgement:} The research of  Y. Li and S. Zhu was supported in part
by National Natural Science Foundation of China under grant 11231006 and Natural Science Foundation of Shanghai under grant 14ZR1423100. Y. Li was also supported by Shanghai Committee of Science and Technology under grant 15XD1502300. S. Zhu was also supported by China Scholarship Council. The research of R. Pan was partially supported
by National Science Foundation under grant DMS-1108994.

\end{document}